\documentclass{amsart}

\usepackage[utf8]{inputenc}
\usepackage[english]{babel}

\usepackage{amsmath, amssymb,mathtools}
\usepackage{dsfont}
\usepackage{cancel}
\usepackage{xfrac}		
\usepackage{mathdots}	
\usepackage{thmtools} 	
\usepackage{tcolorbox}	

\newtheorem{theorem}{Theorem}

\newtheorem{lemma}{Lemma}

\makeatletter
\renewcommand*\env@matrix[1][\arraystretch]{%
  \edef\arraystretch{#1}%
  \hskip -\arraycolsep
  \let\@ifnextchar\new@ifnextchar
  \array{*\c@MaxMatrixCols c}}
\makeatother

\usepackage{graphicx}
\usepackage{subcaption}
\usepackage{pgfplots}	

\usepackage{color}

\usepackage{tikz}
\usetikzlibrary{calc}
\usetikzlibrary{decorations.markings}
\usetikzlibrary{shapes.geometric,positioning}
\usetikzlibrary{arrows}
\tikzset{vertex/.style={minimum size=2mm,circle,fill=black,draw,inner sep=0pt},
         decoration={markings,mark=at position .5 with {\arrow[black,thick]{stealth}}}}
\pgfplotsset{compat=1.10}
\usepgfplotslibrary{fillbetween}
\usetikzlibrary{backgrounds}
\usetikzlibrary{patterns}

\newcommand{\R}{ \mathbb{R} }

\newcommand{\1}{\mathds{1}}

\renewcommand{\Re}[1]{\operatorname{Re}\{#1\}}
\renewcommand{\Im}[1]{\operatorname{Im}\{#1\}}
\newcommand{\ii}{\mathrm{i}}

\newcommand{\inner}[2]{\langle #1,#2 \rangle}
\newcommand{\norme}[1]{\left\lVert #1 \right\rVert}
\newcommand{\abs}[1]{\left\vert #1 \right\vert}
\newcommand{\e}[1]{\mathrm{e}^{#1}}

\renewcommand{\div}[1]{ {\rm div}\left(#1\right) }
\newcommand{\grad}[1]{\nabla #1}
\newcommand{\dpart}[2]{\frac{\partial #1}{\partial #2}}
\newcommand{\ddpart}[2]{\frac{\partial^2 #1}{\partial #2^2}}
\newcommand{\dpp}[2]{\frac{\partial}{\partial #2}\left(#1\right)}
\newcommand{\dn}[1]{\grad{#1} \cdot \hat{n}}

\newcommand{\bigO}{\mathcal{O}}

\renewcommand{\t}{\top}

\newtheorem{rem}[theorem]{Remark}

\newcommand{\Lim}[1]{\raisebox{0.5ex}{\scalebox{0.8}{$\displaystyle \lim_{#1}\;$}}}

\newcommand{\PH}[1]{{\color{black} #1}}

\newcommand{\JS}[1]{{\color{black} #1}}

\title{{Optimization of bathymetry for long waves with small amplitude}}
\author{Pierre-Henri COCQUET$^\star$}
\address{$^\star$ Universit\'e de La R\'eunion, Laboratoire PIMENT, 117 Avenue du G\'en\'eral Ailleret, 97430 Le Tampon, France}
\email{pierre-henri.cocquet@univ-reunion.fr}
\author{Sebasti\'an RIFFO$^\dagger$}
\address{$^\dagger$ CEREMADE, CNRS, UMR 7534, Universit\'e Paris-Dauphine, PSL University}
\email{sebastian.reyes-riffo@dauphine.eu}
\author{Julien Salomon$^*$}
\address{$^*$ INRIA  Paris,  ANGE  Project-Team,  75589  Paris  Cedex  12,  France
   and Sorbonne Universit\'e, CNRS, Laboratoire Jacques-Louis Lions, 75005 Paris, France}
\email{julien.salomon@inria.fr}

\begin{document}

\maketitle

\begin{abstract}
{
This paper deals with bathymetry-oriented optimization in the case of long waves with small amplitude. Under these two assumptions, the free-surface incompressible Navier-Stokes  system can be written as a wave equation where the bathymetry appears as a parameter in the spatial operator. Looking then for time-harmonic fields and writing the bottom topography as a perturbation of a flat bottom, we end up with a heterogeneous Helmholtz equation with impedance boundary condition. In this way, we study some PDE-constrained optimization problem for a Helmholtz equation in heterogeneous media whose coefficients are only bounded with bounded variation. We provide necessary condition for a general cost function to have at least one optimal solution. We also prove the convergence of a finite element approximation of the solution to  the considered Helmholtz equation as well as the convergence of discrete optimum toward the continuous ones. We end this paper with some numerical experiments to illustrate the theoretical results and show that some of their assumptions could actually be removed.
}
\end{abstract}


\section{Introduction}
Despite the fact that the bathymetry can be inaccurately known in many situations, wave propagation models strongly depend on this parameter to capture the flow behavior, which emphasize the importance of studying inverse problems concerning its reconstruction from free surface flows. 
In recent years a considerable literature has grown up around this subject. A review from Sellier identifies different techniques applied for bathymetry reconstruction \cite[Section 4.2]{Sellier2016}, which rely mostly on the derivation of an explicit formula for the bathymetry, numerical resolution of a governing system or data assimilation methods \cite{Honnorat2009,Beach-Wizard}. 

\JS{An} alternative is to use the bathymetry as control variable of a PDE-constrained optimization problem, an approach used in coastal engineering due to mechanical constraints associated with building structures and their interaction with sea waves. For instance, among the several aspects to consider when designing a harbor, building defense structures is essential to protect it against wave impact. These 
can be optimized to locally minimize the wave energy, by studying its interaction with the reflected waves \cite{IAMB}. Bouharguane and Mohammadi \cite{MB2,MB1} consider a time-dependent approach to study the evolution of sand motion at the seabed, which could also allow these structures to change in time. In this case, the proposed functionals are minimized using sensitivity analysis, a technique broadly applied in geosciences. 
From a mathematical point of view, the solving of these kinds of problem is mostly numerical. 
A theoretical approach applied to the modeling of surfing pools can be found in \cite{DB-KdV,NDZ}, where the goal is to maximize locally the energy of the prescribed wave. The former proposes to determine a bathymetry, whereas the latter sets the shape and displacement of an underwater object along a constant depth.

In this paper, we address the determination of a bathymetry from an optimization problem, where a reformulation of the Helmholtz equation acts as a constraint. Even though this equation is limited to describe waves of small amplitude, it is often used in engineering due to its simplicity, which leads to explicit solutions when a flat bathymetry is assumed. 
{
To obtain such a formulation, {we rely} on two asymptotic approximations of the free-surface incompressible Navier-Stokes equations. The first one is based on a 
long-wave theory approach and reduces the Navier-Stokes system to the Saint-Venant equations. The second one \JS{considers} waves of small amplitude from which the Saint-Venant model 
can be approximated by a wave-equation involving the bathymetry in its spatial operator. {It is finally} when considering time-harmonic solution of this wave equation that we get a Helmholtz equation with spatially-varying coefficients. 
Regarding the assumptions on the bathymetry to be optimized, {we assume the latter to be} a perturbation of a flat bottom with a compactly supported perturbation
which can thus be seen as a scatterer. Since we wish to be as closed to real-world applications as possible, {we also assume} that the bottom topography is not smooth and, for instance, can be discontinuous. We therefore end up with a constraint equation given by a time-harmonic wave equation, namely a Helmholtz equation, with non-smooth coefficients.  

It is worth noting that our bathymetry optimization problem aims at finding some parameters in our PDE that minimize a given cost function and can thus be seen 
as a parametric optimization problem (see e.g. \cite{Kunish_book_2012}, \cite{allaire2007conception}, \cite{haslinger2003introduction}). Similar optimization problems 
can also be encountered when trying to identify some parameters in the PDE from measurements (see e.g. \cite{Chen_BV_1999},\cite{beretta2018reconstruction}). Nevertheless, all 
the aforementioned references deals with real elliptic and coercive problems. Since the Helmholtz equation is unfortunately a complex and non-coercive PDE, these results do not apply.

We also emphasize that the PDE-constrained optimization problem studied in the present paper falls into the class of so-called topology optimization problems. For practical applications involving Helmholtz-like equation as constraints, we refer to \cite{wadbro2010shape},\cite{bernland2018acoustic} where the shape of an acoustic horn is optimized to have better transmission efficiency and to \cite{jensen2005topology},\cite{christiansen2019acoustic},\cite{christiansen2019designing} for the topology optimization of photonic crystals where several different cost functions are considered.
Although there is a lot of applied and numerical studies of topology optimization problems involving Helmholtz equation, there are only few theoretical studies as pointed out in~\cite[p. 2]{Haslinger_2015}. 

Regarding the theoretical results from~\cite{Haslinger_2015}, {the authors} proved existence of optimal solution to their PDE-constrained optimization problem 
as well as the convergence of the discrete optimum toward the continuous ones. It is worth noting that, in \JS{this paper}, { a relative permittivity is considered as optimization parameter} 
and that the latter appears as a multiplication operator in the Helmholtz differential operator. Since in the present study the bathymetry is assumed to be non-smooth and is involved in the principal part of our heterogeneous Helmholtz equation, we can not rely on the theoretical results proved in~\cite{Haslinger_2015} to study our optimization problem. 
 }

{This paper is organized as follows: Section \ref{sec:wave-model} presents the two approximations of the free-surface incompressible Navier-Stokes system, namely the long-wave theory approach and next the reduction to waves with small amplitude, that lead us to consider a Helmholtz equation in heterogeneous media where the bathymetry} 
plays the role of a scatterer.
Under suitable assumptions on the cost functional and the admissible set of bathymetries, in Section \ref{sec:optimization} we are able to prove the continuity of the control-to-state mapping and the existence of an optimal solution, in addition to the continuity and boundedness of the resulting wave presented in Section \ref{sec:continuity}. The discrete optimization problem is discussed in Section \ref{sec:disc_pb}, studying the convergence to the discrete optimal solution as well as the convergence of a finite element approximation. Finally, we present some numerical results in Section \ref{sec:numerics}.


\section{Derivation of the wave model}\label{sec:wave-model}
We start from the Navier-Stokes equations to derive the governing PDE. However, due to its complexity, we introduce two approximations \cite{LM}: a small relative depth (\textit{Long wave theory}) combined with an infinitesimal wave amplitude (\textit{Small amplitude wave theory}). An asymptotic analysis on the relative depth shows that the vertical component of the depth-averaged velocity is negligible, obtaining the Saint-Venant equations. After neglecting its convective inertia terms and linearizing around the sea level, it results in a wave equation which depends on the bathymetry. Since a variable sea bottom can be seen as an obstacle, we reformulate the equation as a \textit{Scattering problem} involving the Helmholtz equation.

\subsection{From Navier-Stokes system to Saint-Venant equations}

For $t\geq 0$, we define the time-dependent region 
\[
\Omega_t = \{(x,z)\in \Omega\times\R\; \vert\; -z_b(x) \leq z \leq \eta(x,t)\} 
\]
where $\Omega$ is a bounded open set with Lipschitz boundary, $\eta(x,t)$ represents the water level and $-z_b(x)$ is the bathymetry or bottom topography, a time independent and negative function. The water height is denoted by $h = \eta +z_b$.\\

\begin{center}
\begin{tikzpicture}[>=stealth]
	\def\H{-2.5}		
	\def\L{10}   	
	\def\A{1}		
	
 	\draw[->,dotted] (0,0)  -- (\L,0) node [below] {$x$}; 
	\draw[->,dotted] (0,\H) --(0,\A) node[right] {$z$};	
	
	\draw[blue, x=0.01388cm, y=1cm,
		declare function={
		wave(\t)= 0.2 +0.5*cos(\t);
		}]
	plot [domain=0:360*2, samples=144, smooth] (\x,{wave(\x)});	
	\node[above left,blue] at (\L,0.8) {\small{Free surface}};

	\def\v{5.3}	
	\def\diff{0.1}	
	\def\utab{0.65}
	\def\ltab{\H*0.75}
    	\draw[<->] (\v,\utab) -- (\v,0);	
    	\draw[<->] (\v,0) -- (\v,\ltab);	
	\node[left] at (\v,\utab/2){$\eta(x,t)$};
	\node[left] at (\v,\ltab/2){$-z_b(x)$};
		
	\def\vh{6.5}
	\pgfmathsetmacro\h{\utab+\ltab};
	\draw[dashed,red] (\v,\utab) -- (\vh,\utab);		
	\draw[dashed,red] (\v,\ltab) -- (\vh,\ltab);		
	\draw[<->,red] (\vh,\utab) -- (\vh,\ltab);				
	\node[left] at (\vh,\h/2){$h$};		
		
	\draw[decorate,decoration={random steps,segment length=4mm},thick](0,\H) -- (\v,\ltab);
	\draw[decorate,decoration={random steps,segment length=4mm},thick](\v,\ltab) -- (\L,\H/3);
	\node[below left] at (\L,0.75*\H) {\small{Bottom}};
\end{tikzpicture}
\end{center}

In what follows, we consider an incompressible fluid of constant density (assumed to be equal to 1), governed by the Navier-Stokes system
\begin{equation}\left\{
\begin{aligned}
\dpart{\mathbf{u}}{t} + \left(\mathbf{u}\cdot \nabla\right)\mathbf{u} 
	&= \div{\sigma_{T}} +\mathbf{g} 
	&&\textrm{ in }\Omega_t,\\
\div{\mathbf{u}} 
	&= 0 
	&&\textrm{ in }\Omega_t,\\
\mathbf{u} 
	&= \mathbf{u}_0
	&&\textrm{ in }\Omega_0,
\end{aligned}\right. \label{Euler}
\end{equation}
where $\mathbf{u} = (u,v,w)^{\t}$ denotes the velocity of the fluid, $\mathbf{g} = (0,0,-g)^\t$ is the gravity and $\sigma_{T}$ is the total stress tensor, given by 
\[
\sigma_{T} = -p\mathbb{I} +\mu\left(\grad{\mathbf{u}}+ \grad{\mathbf{u}}^\t\right)
\]
with $p$ the pressure and $\mu$ the coefficient of viscosity.

To complete \eqref{Euler}, we require suitable boundary conditions. Given the outward normals
\[
n_s = \dfrac{1}{\sqrt{1+\abs{\grad{\eta}}^2}}
		\begin{pmatrix}-\grad{\eta}\\1 \end{pmatrix},\;
n_b = \dfrac{1}{\sqrt{1+\abs{\grad{z_b}}^2}}
		\begin{pmatrix} \grad{z_b} \\ 1\end{pmatrix},
\]	
to the free surface and bottom, respectively, we recall that the velocity of the two must be equal to that of the fluid: 
\begin{equation}\left\{
\begin{aligned}
\dpart{\eta}{t} -\mathbf{u}\cdot n_s 
	&= 0 
	&&\textrm{on } (x,\eta(x,t),t),\\
\mathbf{u}\cdot n_b
	&= 0 
	&&\textrm{on } (x,-z_b(x),t).
\end{aligned} \right.\label{flux}
\end{equation}
On the other hand, the stress at the free surface is continuous, whereas at the bottom we assume a no-slip condition
\begin{equation}\left\{
\begin{aligned}
\sigma_{T}\cdot n_s &= -p_a n_s 			
	&&\textrm{on } (x,\eta(x,t),t),  \\
(\sigma_{T} n_b)\cdot t_b &= 0 
	&&\textrm{on } (x,-z_b(x),t),
\end{aligned}\right. \label{stress}
\end{equation}
with $p_a$ the atmospheric pressure and $t_b$ an unitary tangent vector to $n_b$.

A long wave theory approach can then be developed to approximate the previous model by a Saint-Venant system~\cite{Perthame}.  Denoting by $H$ the relative depth and $L$ the characteristic dimension along the horizontal axis, this approach is based on the approximation $\varepsilon := \dfrac{H}{L}\ll 1$, leading to a hydrostatic pressure law for the non-dimensionalized Navier-Stokes system, and a vertical integration of the remaining equations. For the sake of completeness, details of this derivation in our case are given in Appendix. For a two-dimensional system~\eqref{Euler},  the resulting system is then
\begin{align}
\dpart{\eta}{t}\sqrt{1+(\varepsilon\delta)^2 \abs{\dpart{\eta}{x}}^2} +\dpart{(h_{\delta}\overline{u})}{x} 
	&= 0 \label{mass-av2} \\
\dpart{(h_{\delta}\overline{u})}{t} +\delta\dpart{(h_{\delta}\overline{u}^2)}{x}
	&= -h_{\delta}\dpart{\eta}{x} +\delta u(x,\delta\eta,t)\dpart{\eta}{t}\bigg(\sqrt{1+(\varepsilon\delta)^2\abs{\dpart{\eta}{x}}^2}-1\bigg) \nonumber\\
	&\qquad +\bigO(\varepsilon) +\bigO(\delta\varepsilon),\label{momentum-av2} 
\end{align}
where   $\delta := \dfrac{A}{H}$, $h_{\delta} = \delta\eta +z_b$ and $\overline{u}(x,t) := \dfrac{1}{h_{\delta}(x,t)}\int_{-z_b}^{\delta\eta}u(x,z,t)dz$. If 
$\varepsilon\rightarrow 0$, we recover the classical derivation of the one-dimensional Saint-Venant equations.

\subsection{Small amplitudes}
With respect to the classical Saint-Venant formulation, passing to the limit $\delta\rightarrow 0$ is equivalent to neglect the convective acceleration terms and linearize the system (\ref{mass-av2}-\ref{momentum-av2}) around the sea level $\eta = 0$. In order to do so, we rewrite the derivatives as
\[
\dpart{(h_{\delta}\overline{u})}{t} 
	= h_{\delta}\dpart{\overline{u}}{t} +\delta\dpart{\eta}{t}\overline{u},\;
\dpart{(h_{\delta}\overline{u})}{x}
	= \delta\dpart{(\eta\overline{u})}{x} +\dpart{(z_b\overline{u})}{x},
\] 
and then, taking $\varepsilon,\delta\rightarrow 0$ in (\ref{mass-av2}-\ref{momentum-av2}) yields
\begin{equation*}\left\{
\begin{aligned}
\dpart{\eta}{t} +\dpart{(z_b\overline{u})}{x}
	&= 0, \\
-\dpart{(z_b\overline{u})}{t} +z_b\dpart{\eta}{x}
	&= 0. 
\end{aligned}\right.
\end{equation*}
Finally, after deriving the first equation with respect to $t$ and replacing the second into the new expression, we obtain the wave equation for a variable bathymetry. 
All the previous computations hold for the three-dimensional system \eqref{Euler}. In this case, we obtain
\begin{equation}\label{wave-eq}
	\ddpart{\eta}{t} -\div{gz_b\grad{\eta}} = 0. 
\end{equation}

\subsection{Helmholtz formulation}
Equation \eqref{wave-eq} defines a time-harmonic field, whose solution has the form $\eta(x,t)=\Re{\psi_{tot}(x)\e{-\ii\omega t}}$, where the amplitude $\psi_{tot}$ satisfies 
\begin{equation} \label{ampli-helmholtz1}
	\omega^2 \psi_{tot} +\div{gz_b\grad{\psi_{tot}}} = 0.
\end{equation}

We wish to rewrite the equation above as a scattering problem. Since a variable bottom $z_b(x):= z_0 +\delta z_b(x)$ (with $z_0$ a constant describing a flat bathymetry and $\delta z_b$ a perturbation term) can be considered as an obstacle, we thus assume that $\delta z_b$ has a compact support in $\Omega$ and that $\psi_{tot}$ satisfies the so-called Sommerfeld radiation condition. In a bounded domain as $\Omega$, we impose the latter thanks to an impedance boundary condition (also known as first-order absorbing boundary condition), which ensures the existence and uniqueness of the solution \cite[p. 108]{Nedelec}. We then reformulate \eqref{ampli-helmholtz1} as 
\begin{equation}\left\{
\begin{aligned}
	\div{(1+q) \grad{\psi_{tot}}} +k_0^2\psi_{tot}  &= 0 && \textrm{ in }\Omega,\\
	\dn{(\psi_{tot}-\psi_{0})} -\ii k_0(\psi_{tot}-\psi_{0}) &= 0 && \textrm{ on }\partial\Omega, 
\end{aligned}\right. 
\label{ampli-helmholtz}
\end{equation}
where \JS{we have introduced the parameter $q(x)\vcentcolon=\frac{\delta z_b(x)}{z_0}$ which is assumed to be compactly supported in $\Omega$}, $k_0 \vcentcolon=\frac{\omega}{\sqrt{gz_0}}$, $\hat{n}$ the unit normal to $\partial \Omega$ and $\psi_{0}(x)=\mathrm{e}^{\ii k_0x\cdot \vec{d}}$ is an incident plane wave propagating in the direction $\vec{d}$ (such that $|\vec{d}|=1$).

Decomposing the total wave as $\psi_{tot} = \psi_{0} + \psi_{sc}$, where $\psi_{sc}$ represents an unknown scattered wave, we obtain the Helmholtz formulation
\begin{equation}\label{eq:ampli-helmholtz_sc}
\left\{
\begin{aligned}
	\div{(1+q) \grad{\psi}_{sc}} +k_0^2\psi_{sc}  &= -\div{q \grad{\psi_0}} && \textrm{ in }\Omega,\\
	\dn{\psi_{sc}} -\ii k_0\psi_{sc} &= 0 && \textrm{ on }\partial\Omega. 
\end{aligned}\right. 
\end{equation}
Its structure will be useful to prove the existence of a minimizer for a PDE-constrained functional, as discussed in the next section.

\section{Description of the optimization problem}\label{sec:optimization}
We are interested in studying the problem of a cost functional constrained by the weak formulation of a Helmholtz equation. The latter intends to generalize the equations considered so far, whereas the former indirectly affects the choice of the set of admissible controls. These can be discontinuous since they are included in the space of functions of bounded variations. In this framework, we treat the continuity and regularity of the associated control-to-state mapping, and the existence of an optimal solution to the optimization problem.

\subsection{Weak formulation}\label{sub:weak_form}

Let $\Omega\subset\R^2$ be a bounded open set with Lipschitz boundary. We consider the following general Helmholtz equation 
\begin{equation}\label{eq:Helmholtz}
\left\{
\begin{aligned}
	-\div{(1+q) \grad{\psi}} -k_0^2\psi  &= \div{q \grad{\psi_0}} && \textrm{ in }\Omega,\\
	(1+q)\dn{\psi} -\ii k_0\psi &= g -q\dn{\psi_0} && \textrm{ on }\partial\Omega, 
\end{aligned}\right. 
\end{equation}
where $g$ is a source term. We assume that $q\in L^{\infty}(\Omega)$ and that there exists $\alpha>0$ such that 
\begin{equation}\label{eq:hyp_q}
\mathrm{a.e.}\ x\in\Omega,\ 1+q(x)\geq \alpha.
\end{equation} 

\begin{rem}\label{rem:Link_general_Helmholtz_scatt}
Here we have generalized the models described in the previous section: if $q$ has a fixed compact support in $\Omega$, we have that the total wave $\psi_{tot}$ satisfying \eqref{ampli-helmholtz} is a solution to (\ref{eq:Helmholtz}) with $g=\dn{\psi_{0}} -\ii k_0\psi_{0}$ and no volumic right-hand side; whereas the scattered wave $\psi_{sc}$ satisfying \eqref{eq:ampli-helmholtz_sc} is a solution to \eqref{eq:Helmholtz} with $g=0$. All the proofs obtained in this broader setting still hold true for both problems.
\end{rem}

A weak formulation for \eqref{eq:Helmholtz} is given by 
\begin{equation}\label{eq:FV_Helmholtz}
a(q;\psi,\phi)=b(q;\phi),\ \forall \phi\in H^1(\Omega),
\end{equation}
where 
\begin{align}
a(q;\psi,\phi) 
	&\vcentcolon=\int_\Omega \left((1+q)\nabla\psi \cdot \nabla\overline{\phi}-k_0^2 \psi\overline{\phi}\right)\, dx-\ii k_0\int_{\partial\Omega} \psi\overline{\phi}\, d\sigma,\label{def:a}\\
b(q;\phi) 
	&\vcentcolon= -\int_\Omega q \nabla \psi_0\cdot \nabla\overline{\phi}\, dx +\inner{g}{\overline{\phi}}_{H^{-1/2},H^{1/2}}. \nonumber
\end{align}
Note that, thanks to Cauchy-Schwarz inequality, the sesquilinear form $a$ is continuous
\begin{align}
|a(q;\psi,\phi)|	&\leq C(\Omega,q,\alpha) (1+\norme{q}_{L^\infty(\Omega)})  \norme{\psi}_{1,k_0}\norme{\phi}_{1,k_0},\nonumber\\
\norme{\psi}^2_{1,k_0}&\vcentcolon=k_0^2\norme{\psi}_{L^2(\Omega)}^2+\alpha\norme{\nabla \psi}^2_{L^2(\Omega)},\nonumber
\end{align} 
where $C(\Omega,q,\alpha)>0$ is a generic constant. In addition, taking $\phi=\psi$ in the definition of $a$, it satisfies a G\aa rding inequality 
\begin{equation}\label{eq:Garding_ineq}
\Re{a(q;\psi,\psi)}+2k_0^2\norme{\psi}^2_{L^2(\Omega)} \geq \norme{\psi}^2_{1,k_0},
\end{equation}
and the well-posedness of Problem \eqref{eq:FV_Helmholtz} follows from the Fredholm Alternative. Finally, uniqueness holds for any $q\in L^\infty(\Omega)$ satisfying \eqref{eq:hyp_q} owning to \cite[Theorems 2.1, 2.4]{Graham_variable_2018}.


\subsection{Continuous optimization problem}

We are interested in solving the next PDE-constrained optimization problem
\begin{equation}\label{eq:Optim_pbm}
\begin{aligned}
\textrm{minimize } 	& J(q,\psi),\\
\textrm{subject to }	& (q,\psi)\in U_\Lambda\times H^1(\Omega), \textrm{ where } \psi \textrm{ satisfies } (\ref{eq:FV_Helmholtz}).
\end{aligned}
\end{equation} 
We now define \JS{the set $U_\Lambda$} of admissible $q$. We wish to find optimal $q$ that can have discontinuities and we thus cannot look for $q$ in some Sobolev spaces that are continuously embedded into $C^0(\overline{\Omega})$, even if such regularity is useful for proving \JS{existence of minimizers} (see e.g. \cite[Chapter VI]{Kunish_book_2012}, \cite[Theorem 4.1]{Bastide_2018}). 
To be able to find \JS{an} optimal $q$ satisfying (\ref{eq:hyp_q}) and having possible discontinuities, we follow \cite{Chen_BV_1999} and introduce the following set 
\[
U_\Lambda=\left\{ q\in BV(\Omega)\ \left| \ \alpha-1\leq q(x)\leq \Lambda\ a.e. \ x\in\Omega \right. \right\}.
\]
Above $\Lambda\geq \max\{\alpha -1,0\}$ and $BV(\Omega)$ is the set of functions with bounded variations \cite{BV_properties}, that is functions whose distributional gradient \JS{belongs} to the set $\mathcal{M}_\mathrm{b}(\Omega,\R^N)$ of bounded Radon \JS{measures}. Note that the piecewise constant functions over $\Omega$ belong to $U_\Lambda$.

Some useful properties of $BV(\Omega)$ can be found in \cite{BV_properties} and are recalled below for the sake of completeness. This is a Banach space for the norm (see \cite[p. 120, Proposition 3.2]{BV_properties})
\[
\norme{q}_{BV(\Omega)}\vcentcolon=\norme{q}_{L^1(\Omega)} + |D q|(\Omega),
\]
where $D$ is the distributional gradient and 
\begin{equation}\label{def:Dq}
|D q|(\Omega)=\sup\left\{\int_\Omega q\,  \div{\varphi}\, dx\ \left|\ \varphi\in \mathcal{C}^1_\mathrm{c}(\Omega,\R^2)\ \mathrm{and}\ \norme{\varphi}_{L^\infty(\Omega)}\leq 1 \right.\right\},
\end{equation}
is the variation of $q$ (see \cite[p. 119, Definition 3.4]{BV_properties}).

The weak$^*$ convergence in $BV(\Omega)$, denoted by
\[
q_n\rightharpoonup q,\ \mathrm{weak}^* \ \mathrm{in}\ BV(\Omega),
\]
means that 
\[
q_n\rightarrow q\ \mathrm{in}\ L^1(\Omega)\ \mathrm{and}\ D q_n \rightharpoonup D q\ \mathrm{in} \ \mathcal{M}_\mathrm{b}(\Omega,\R^N).
\]
Also, in a two-dimensional setting, the continuous embedding 
$ BV(\Omega)\subset L^1(\Omega) $ is compact. We finally recall that the application $q\in BV(\Omega)\mapsto |Dq|(\Omega)\in \R^+$ is lower semi-continuous with respect to the weak$^*$ topology of $BV$.
Hence, for any sequence $q_n\rightharpoonup q$ in $BV(\Omega)$, one has
$$
|Dq|(\Omega)\leq \liminf_{n\to+\infty} |Dq_n|(\Omega).
$$

The set $U_{\Lambda}$ is a closed, weakly$^*$ closed and convex subset of $BV(\Omega)$. However, since its elements are not necessarily bounded in the $BV$-norm, \JS{we add a penalizing distributional gradient term to the cost functional $J(q,\psi)$ to prove the existence of a minimizer to Problem \eqref{eq:Optim_pbm}.
In this way, we introduce the set of admissible parameters}
\[
U_{\Lambda,\kappa}=\left\{ q\in U_{\Lambda}\ | \ |Dq|(\Omega)\leq \kappa\right\}
\]
which also possesses the aforementioned properties. 
\JS{Note that choosing $U_{\Lambda}$ or $U_{\Lambda,\kappa}$} affects the convergence analysis of the discrete optimization problem, topic discussed in Section \ref{sec:disc_pb}.


\begin{rem}\label{rem:U_epsilon}
In this paper, we are interested in computing either the total wave satisfying \eqref{ampli-helmholtz} or the scattered wave solution to Equation \eqref{eq:ampli-helmholtz_sc}. 
Since this \JS{requires} to work with $q$ having a fixed compact support in $\Omega$, we also introduce
the following set of admissible parameters  
\[
\widetilde{U}_{\varepsilon}\vcentcolon=\left\{ q\in U\ | q(x)=0\ \mathrm{for\ a.e}\ x\in\mathcal{O}_\varepsilon \right\},\ \mathcal{O}_\varepsilon=\left\{x\in\Omega\ |\ \mathrm{dist}(x,\partial\Omega)\leq \varepsilon \right\},
\]
which is a set of bounded functions with bounded variations that have a fixed support in $\Omega$.
We emphasize that \JS{this set} is a convex, closed and weak-$*$ closed subset of $BV(\Omega)$. As a \JS{consequence}, all the theorems we are going to prove also hold for this set of admissible parameters.
\end{rem}
\subsection{Continuity of the control-to-state mapping}
In this section, \JS{we establish} the continuity of the application $q\in U\mapsto \psi(q)\in H^1(\Omega)$ where $\psi(q)$ satisfies Problem (\ref{eq:FV_Helmholtz}).
We assume that $U\subset BV(\Omega)$ is \JS{a given weakly$^*$ closed set satisfying} 
\[
\forall q\in U,\ \mathrm{a.e.\ }x\in\Omega,\ \alpha-1\leq q(x)\leq \Lambda. 
\] 
Note that both $U_\Lambda$, $U_{\Lambda,\kappa}$ and $\widetilde{U}_{\varepsilon}$ (see Remark \ref{rem:U_epsilon}) also satisfy these two assumptions.
The next result consider the dependance of the stability constant with respect to the optimization parameter $q$.

\begin{theorem}\label{thm:Uniform_H1_bound}
Assume that $q\in U$ and $\psi\in H^1(\Omega)$. Then there exists a constant $C_{\mathrm{s}}(k_0)>0$ that does not depend on $q$ such that 
\begin{equation}\label{eq:inf_sup_cont}
\norme{\psi}_{1,k_0}\leq C_{\mathrm{s}}(k_0) \sup_{\norme{\phi}_{1,k_0}=1}|a(q;\psi,\phi)|,
\end{equation}
where the constant $C_{\mathrm{s}}(k_0)>0$ only depend on the wavenumber and on $\Omega$. In addition, if $\psi$ is the solution to (\ref{eq:FV_Helmholtz}) then it satisfies the bound
\begin{equation}\label{eq:Uniform_H1_bound}
\norme{\psi}_{1,k_0}\leq C_{\mathrm{s}}(k_0)C(\Omega)\max\{k_0^{-1},\alpha^{-1/2}\}\left(\norme{q}_{L^\infty(\Omega)}\norme{\nabla \psi_0}_{L^2(\Omega)}+\norme{g}_{H^{-1/2}(\partial\Omega)} \right),
\end{equation}
where $C(\Omega)>0$ only depends on the domain. 
\end{theorem}

\begin{proof}
\JS{The existence and uniqueness of a solution to Problem \eqref{eq:FV_Helmholtz} follows from \cite[Theorems 2.1, 2.4]{Graham_variable_2018}.}

The proof of \eqref{eq:inf_sup_cont} proceed by contradiction assuming \JS{this inequality to be} false.  Therefore, we suppose there exist sequences $(q_n)_n\subset U$ and $(\psi_n)_n\subset H^1(\Omega)$
such that $\norme{q_n}_{BV(\Omega)}\leq M$, $\norme{\psi_n}_{1,k_0}=1$ and
\begin{equation}\label{eq:Contradiction_stab}
\lim_{n\to+\infty}\sup_{\norme{\phi}_{1,k_0}=1}|a(q_n;\psi_n,\phi)|=0.
\end{equation}
The compactness of the embeddings $BV(\Omega)\subset L^1(\Omega)$ and $H^1(\Omega)\subset L^2(\Omega)$ yields the existence of a subsequence (still denoted $(q_n,\psi_n)$) such that  
\begin{equation}\label{eq:Convergence_contradiction}
\psi_n\rightharpoonup \psi_\infty\ \mathrm{in}\ H^1(\Omega),\ 
\psi_n\rightarrow \psi_\infty\ \mathrm{in}\ L^2(\Omega)\ \mathrm{and}\ q_n\rightarrow q_\infty\in U\ \mathrm{in}\ L^1(\Omega).
\end{equation}
Compactness of the trace operator implies that $\displaystyle{\lim_{n\to+\infty}\psi_n|_{\partial\Omega}=\psi_\infty|_{\partial\Omega}}$ holds strongly in $L^2(\partial\Omega)$ and thus, from \eqref{eq:Convergence_contradiction} we get
\begin{align*}
\lim_{n\to+\infty} \int_\Omega k_0^2 \psi_n\overline{\phi}\, dx+\ii k_0\int_{\partial\Omega} \psi_n\overline{\phi}\, d\sigma
	&=\!\int_\Omega k_0^2 \psi_\infty\overline{\phi}\, dx+\ii k_0\int_{\partial\Omega} \psi_\infty\overline{\phi}\, d\sigma,\ \forall\ v\in H^1(\Omega), \\
\lim_{n\to+\infty}\int_\Omega \nabla \psi_n\cdot \nabla\overline{\phi}\, dx
	&=\!\int_\Omega \nabla \psi_\infty\cdot \nabla\overline{\phi}\, dx.
\end{align*}
\JS{We now pass to the limit in the term of $a$ that involves $q_n$, see~\eqref{def:a}}. 
We start from 
\begin{align*}
(q_n\nabla \psi_n,\nabla\overline{\phi})_{L^2(\Omega)}-(q_\infty\nabla \psi_\infty,\nabla\overline{\phi})_{L^2(\Omega)}
	&= ((q_n-q_\infty)\nabla \psi_n,\nabla \overline{\phi})_{L^2(\Omega)}\\
	&\qquad +(q_\infty\nabla(\psi_n-\psi_\infty),\nabla\overline{\phi})_{L^2(\Omega)},
\end{align*}
and use Cauchy-Schwarz inequality to get
\begin{align*} 
\int_\Omega q_n \nabla \psi_n\cdot & \nabla\overline{\phi}\, dx 
	-\int_\Omega q_\infty\nabla \psi_\infty\cdot \nabla\overline{\phi}\,dx \\		&\leq \left|((q_n-q_\infty)\nabla \psi_n,\nabla \overline{\phi})_{L^2(\Omega)} \right| 
	+\left|(q_\infty\nabla(\psi_n-\psi_\infty),\nabla\overline{\phi})_{L^2(\Omega)}\right| \\
	&\leq \norme{\sqrt{|q_n-q_\infty|}\nabla \phi}_{L^2(\Omega)}\norme{\sqrt{|q_n-q_\infty|}\nabla \psi_n}_{L^2(\Omega)} \\
	&\qquad+\left|(q_\infty\nabla(\psi_n-\psi_\infty),\nabla\overline{\phi})_{L^2(\Omega)}\right| \\
	&\leq 2\dfrac{\sqrt{\Lambda}}{\sqrt{\alpha}}\norme{\psi_n}_{1,k_0}  \norme{\sqrt{|q_n-q_\infty|}\nabla \phi}_{L^2(\Omega)}+\left|(\nabla(\psi_n-\psi_\infty),q_\infty\nabla\overline{\phi})_{L^2(\Omega)}\right|.
\end{align*}
The right term above goes to $0$ owning to $q_\infty\in L^\infty(\Omega)$ and \eqref{eq:Convergence_contradiction}. For the other term, since $q_n\to q_\infty$ strongly in $L^1$, 
we can extract another subsequence $(q_{n_k})_k$ such that $q_{n_k}\to q_\infty$ pointwise almost everywhere in $\Omega$. Also, $\sqrt{|q_n-q_\infty|}|\nabla \phi|^2\leq 2\sqrt{\Lambda}|\nabla \phi|^2\in L^1(\Omega)$ and the Lebesgue dominated convergence theorem then yields
\[
\lim_{k\to+\infty}\norme{\sqrt{|q_{n_k}-q_\infty|}\nabla \phi}_{L^2(\Omega)}= 0.
\] 
This gives that (see also \cite[Equation (2.4)]{Chen_BV_1999})
\begin{equation}\label{eq:Convergence_Terme_ordre_2}
\lim_{k\to+\infty} (q_{n_k}\nabla \psi_{n_k},\nabla\overline{\phi})_{L^2(\Omega)}=(q_\infty\nabla \psi_\infty,\nabla\overline{\phi})_{L^2(\Omega)},\ \forall \phi\in H^1(\Omega).
\end{equation}
Finally, gathering \eqref{eq:Convergence_Terme_ordre_2} together with \eqref{eq:Contradiction_stab} yields
\[
0=\lim_{k\to+\infty}a(q_{n_k};\psi_{n_k},\phi)=a(q_\infty,\psi_\infty,\phi),\ \forall \phi\in H^1(\Omega),
\]
and the uniqueness result \cite[Theorems 2.1, 2.4]{Graham_variable_2018} shows that $\psi_\infty=0$ thus the whole sequence actually converges to $0$. To get our contradiction, it remains to show that 
$\norme{\nabla \psi_n}_{L^2(\Omega)}$ converges to $0$ as well. From the G\aa rding inequality \eqref{eq:Garding_ineq}, we have 
\[
\norme{\psi_n}^2_{1,k_0}\leq  \Re{a(q_n;\psi_n,\psi_n)}+2k_0^2\norme{\psi_n}^2_{L^2(\Omega)}
\xrightarrow[n\to+\infty]{} 0,
\]
where we used \eqref{eq:Contradiction_stab} and the strong $L^2$ convergence of $\psi_n$ towards $\psi_\infty=0$. Finally one gets $\Lim{n\to+\infty}\norme{\psi_n}_{1,k_0}=0$ which contradicts $\norme{\psi_n}_{1,k_0}=1$ and gives the desired result.

Applying then \eqref{eq:inf_sup_cont} to the solution to \eqref{eq:FV_Helmholtz} finally yields
\begin{align*}
\norme{\psi}_{1,k_0}
	&\leq C_{\mathrm{s}}(k_0) \sup_{\norme{\phi}_{1,k_0}=1}|a(q;\psi,\phi)|
	\leq C_{\mathrm{s}}(k_0) \sup_{\norme{\phi}_{1,k_0}=1} |b(q;\phi)| \\
	& \leq C_{\mathrm{s}}(k_0)\sup_{\norme{\phi}_{1,k_0}=1}\left( \norme{q}_{L^\infty(\Omega)}\norme{\nabla \psi_0}_{L^2(\Omega)}\norme{\nabla \phi}_{L^2(\Omega)}+\norme{g}_{H^{-1/2}(\partial\Omega)}\norme{\phi}_{H^{1/2}(\partial\Omega)}\right) \\	
	& \leq C_{\mathrm{s}}(k_0)C(\Omega)\max\{k_0^{-1},\alpha^{-1/2}\}\left(\norme{q}_{L^\infty(\Omega)}\norme{\nabla \psi_0}_{L^2(\Omega)}+\norme{g}_{H^{-1/2}(\partial\Omega)} \right),
\end{align*}
where $C(\Omega)>0$ comes from the trace inequality. 
\end{proof}

\begin{rem}
Let us consider a more general version of Problem \eqref{eq:Helmholtz}, given by 
\begin{equation*}
\left\{
\begin{aligned}
	-\div{(1+q) \grad{\psi}} -k_0^2\psi  &= F && \textrm{ in }\Omega,\\
	(1+q)\dn{\psi} -ik_0\psi &= G && \textrm{ on }\partial\Omega. 
\end{aligned}\right. 
\end{equation*}
We emphasize that the estimation of the stability constant $C_{\mathrm{s}}(k_0)$ with respect to the wavenumber have been obtained for $(F,G)\in L^2(\Omega)\times L^2(\partial\Omega)$ for $q=0$ in \cite{Hetmaniuk_2007} and for $q\in \mathrm{Lip}(\Omega)$ satisfying \eqref{eq:hyp_q} in \cite{Barucq_2017,Graham_variable_2018,Graham_variable_2018_2}. Since their proofs rely on Green, Rellich and Morawetz identities, they do not extend to the case $(F,G)\in \left(H^1(\Omega)\right)^\prime\times H^{-1/2}(\partial\Omega)$ but such cases can be tackled as it is done in \cite[p.10, Theorem 2.5]{Esterhazy_2012}. The case of Lipschitz $q$ has been studied in \cite{Brown_Helmholtz_Lip}.
As a result, the dependance of the stability constant with respect to $q$, in the case $q\in U$ and $(F,G)\in \left(H^1(\Omega)\right)^\prime\times H^{-1/2}(\partial\Omega)$, does not seem to have been tackled so far to the best of our knowledge.
\end{rem}
\bigskip

\begin{rem}[$H^1$-bounds for the total and scattered waves]\label{rem:H_1_bounds_total_sc_waves}
From Remark \ref{rem:Link_general_Helmholtz_scatt}, we obtain that the total wave $\psi_{tot}$  and the scattered wave $\psi_{sc}$ are solutions to \eqref{eq:FV_Helmholtz}, with respective right hand sides
\[
b_{tot}(q;\phi)=\int_{\partial\Omega} (\dn{\psi_0} -\ii k_0\psi_0)\overline{\phi}\, d\sigma, \;\;
b_{sc}(q;\phi)=-\int_{\Omega}q\nabla\psi_{0}\cdot\nabla{\overline\phi}\, dx.
\]
As a result of Theorem \ref{thm:Uniform_H1_bound} and the continuity of the trace, we have
\begin{align*}
\norme{\psi_{tot}}_{1,k_0}
	&\leq C(\Omega) C_\mathrm{s}(k_0)k_0 \max\{k_0^{-1},\alpha^{-1/2}\}, \\
\norme{\psi_{sc}}_{1,k_0} 
	&\leq C_\mathrm{s}(k_0)\alpha^{-1/2}\norme{q}_{L^\infty(\Omega)} \norme{\nabla \psi_0}_{L^2(\Omega)}
	\leq k_0C_\mathrm{s}(k_0)\alpha^{-1/2}\norme{q}_{L^\infty(\Omega)} \sqrt{|\Omega|}.
\end{align*}
\end{rem}

%

We can now prove some regularity for the control-to-state mapping. 

\begin{theorem}\label{thm:CTS_mapping}
Let $(q_n)_n\subset U$ be a sequence satisfying $\norme{q_n}_{BV(\Omega)}\leq M$ and whose weak$^*$ limit in $BV(\Omega)$ is denoted by $q_\infty$. Let $(\psi(q_n))_n$ be the sequence of weak solution to Problem (\ref{eq:FV_Helmholtz}). Then $\psi(q_n)$ converges strongly in $H^1(\Omega)$ towards $\psi(q_\infty)$. In other words, the mapping 
\[
q\in (U_\Lambda,\mathrm{weak}^*)\mapsto \psi(q)\in (H^1(\Omega),\mathrm{strong}), 
\]
is continuous.
\end{theorem}

\begin{proof}
Since $\norme{q_n}_{BV(\Omega)}\leq M$ and $(q_n)_n\subset U$, there exists $q_\infty$ such that $q_n\rightharpoonup q_\infty,\ \mathrm{weak}^* \ \mathrm{in}\ BV(\Omega)$.
Using that $U$ is $\mathrm{weak}^*$ closed, we obtain that $q_\infty\in U$. 
Note that the sequence $(\psi(q_n))_n$ of solution to Problem \eqref{eq:FV_Helmholtz} satisfies estimate \eqref{eq:Uniform_H1_bound} uniformly with respect to $n$. As a result, 
\PH{there exists some $\psi_\infty\in H^1(\Omega)$ such that}
the convergences \eqref{eq:Convergence_contradiction} hold. 
\PH{Using then} \eqref{eq:Convergence_Terme_ordre_2}, 
we get that $a(q_n;\psi(q_n),\phi)\to a(q_\infty;\psi_\infty,\phi)$. 


Since $b(q_n,\phi)\to b(q_\infty,\phi)$ for all $\phi\in H^1(\Omega)$, this proves that $a(q_\infty;\psi_\infty,\phi)=b(q;\phi)$ for all $\phi\in H^1(\Omega)$. \JS{Consequently} $\psi_\infty=\psi(q_\infty)$ \PH{owning to the uniqueness of} a weak solution to \eqref{eq:FV_Helmholtz} 
and \PH{we have also proved that} $\psi(q_n)\rightharpoonup \psi(q_\infty)$ in $H^1(\Omega)$.

We now show that $\psi(q_n)\to \psi(q_\infty)$ strongly in $H^1$.
To see this, we start by noting that, up to extracting a subsequence (still denoted \JS{by} $q_n$), we can use \eqref{eq:Convergence_Terme_ordre_2} to get that 
\[
\lim_{n\to+\infty} b(q_n;\psi(q_n))=b(q_\infty;\psi(q_\infty)).
\]
Since $\psi(q_n),\psi(q_\infty)$ satisfy the variational problem (\ref{eq:FV_Helmholtz}), we infer 
\begin{equation}\label{eq:Strong_limit_bilinear_form}
\lim_{n\to+\infty} a(q_n;\psi(q_n),\psi(q_n))=a(q_\infty;\JS{\psi(q_\infty)},\psi(q_\infty)),
\end{equation}
where the whole sequence actually converges owing to the uniqueness of the limit. Using then that $\psi(q_n)\rightharpoonup \psi(q_\infty)$ in $H^1(\Omega)$ together with \eqref{eq:Strong_limit_bilinear_form}, one gets 
\begin{align*}
\norme{\sqrt{1+q_n}\nabla \psi(q_n)}^2_{L^2(\Omega)}
	&= a(q_n;\psi(q_n),\psi(q_n))+k_0\norme{\psi(q_n)}^2_{L^2(\Omega)} +\ii k_0\norme{\psi(q_n)}^2_{L^2(\partial\Omega)} \\
	& \hspace{-1cm} \xrightarrow[n\to+\infty]{}	
	a(q_\infty;\psi(q_\infty),\psi(q_\infty))+k_0\norme{\psi(q_\infty)}^2_{L^2(\Omega)} +\ii k_0\norme{\psi(q_\infty)}^2_{L^2(\partial\Omega)} \\
	&= \norme{\sqrt{1+q_\infty}\nabla \psi(q_\infty)}^2_{L^2(\Omega)}.
\end{align*}

To show that $\Lim{n\to+\infty}\norme{\nabla \psi(q_n)}^2_{L^2(\Omega)}=\norme{\nabla \psi(q_\infty)}^2_{L^2(\Omega)}$, note that 
\[
\nabla \psi(q_n)=\frac{\sqrt{1+q_n}\nabla \psi(q_n)}{\sqrt{1+q_n}}.
\]
Using the same arguments as those to prove (\ref{eq:Convergence_Terme_ordre_2}), we have a subsequence (same notation used) such that $q_n\rightarrow q_\infty$ pointwise a.e. in $\Omega$ and thus $\sqrt{1+q_n}^{-1}\rightarrow \sqrt{1+q_\infty}^{-1}$ pointwise a.e. in $\Omega$. Due to Lebesgue's dominated convergence theorem and $\sqrt{1+q_n}\nabla \psi(q_n)\rightarrow \sqrt{1+q_\infty}\nabla \psi(q_\infty)$ strongly in $L^2(\Omega)$, we have
\[
\nabla \psi(q_n)=\dfrac{\sqrt{1+q_n}\nabla \psi(q_n)}{\sqrt{1+q_n}}\rightarrow \dfrac{\sqrt{1+q_\infty}\nabla \psi(q_\infty)}{\sqrt{1+q_\infty}}=\nabla \psi(q_\infty) \textrm{  strong in}\ L^2(\Omega).
\]

The latter, together with \PH{the weak $H^1$-convergence}
show that $\psi(q_n)\to \psi(q_\infty)$ strongly in $H^1$. 
\end{proof}

\subsection{Existence of optimal solution in $U_{\Lambda}$}
We are now in \JS{a} position to prove the existence of a minimizer to Problem \eqref{eq:Optim_pbm}.

\begin{theorem}\label{thm:existence_min}
Assume that the cost function $(q,\psi)\in U_\Lambda\mapsto J(q,\psi)\in \R$ satisfies:
\begin{itemize}
\item[(A1)] There exists $\beta>0$ {and $J_0$} such that 
$$J(q,\psi)=J_0(q,\psi)+\beta |D q|(\Omega),$$
where $|D q|(\Omega)$ is defined in~\eqref{def:Dq}.
\item[(A2)] $\forall (q,\psi)\in U_\Lambda\times H^1(\Omega)$, $J_0(q,\psi)\geq m>-\infty$.
\item[(A3)] $(q,\psi)\mapsto J_0(q,\psi)$ is lower-semi-continuous with respect to the (weak$^*$,weak) topology of $BV(\Omega)\times H^1(\Omega)$.
\end{itemize} 

Then the optimization problem \eqref{eq:Optim_pbm} has at least one optimal solution in $U_\Lambda\times H^1(\Omega)$.
\end{theorem}

\begin{proof}
The existence of a minimizer to Problem \eqref{eq:Optim_pbm} can be obtained with standard technique by combining Theorem \ref{thm:CTS_mapping} with weak-compactness arguments as done in \cite[Lemma 2.1]{Chen_BV_1999}, \cite[Theorem 4.1]{Bastide_2018} or \cite[Theorem 1]{Haslinger_2015}. We still give the proof for the sake of completeness.

We introduce the following set 
\[
\mathcal{A}=\left\{(q,\psi)\in U_\Lambda\times H^1(\Omega)\ \left|\ a(q;\psi,\phi)=b(q;\phi)\ \forall \phi\in H^1(\Omega) \right.\right\}.
\]
The existence and uniqueness of solution to Problem (\ref{eq:FV_Helmholtz}) ensure that $\mathcal{A}$ is non-empty. In addition, combining \JS{Assumptions}~$(A1)$ and $(A2)$, we obtain that $J(q,\psi)$ is bounded from below on $\mathcal{A}$. We thus have a minimizing sequence $(q_n,\psi_n)\in \mathcal{A}$ such that 
\[
\lim_{n\to +\infty} J(q_n,\psi_n)=\inf_{(q,\psi)\in \mathcal{A}} J(q,\psi).
\] 
Theorem \ref{thm:Uniform_H1_bound} and $(A1)$ then gives that the sequence $(q_n,\psi_n)\in BV(\Omega)\times H^1(\Omega)$ is uniformly bounded with respect to $n$ and thus admits a subsequence that converges towards $(q^*,\psi^*)$ in the (weak$^*$,weak) topology of $BV(\Omega)\times H^1(\Omega)$. Using now Theorem \ref{thm:CTS_mapping} and the weak$^*$ lower semi-continuity of $q\mapsto |Dq|(\Omega)$,
we end up with $(q^*,\psi^*)\in \mathcal{A}$ and 
\[
J(q^*,\psi^*)\leq \liminf_{n\to+\infty} J(q_n,\psi_n)=\inf_{(q,\psi)\in \mathcal{A}} J(q,\psi).
\]
\end{proof}

It is worth noting that \JS{ the penalization term $\beta\norme{q}_{BV(\Omega)}$} has been introduced only to obtain a uniform bound in the $BV$-norm for the minimizing sequence. 

\subsection{Existence of optimal solution in $U_{\Lambda,\kappa}$}
We show here the existence of optimal solution to Problem \eqref{eq:Optim_pbm} for $U=U_{\Lambda,\kappa}$. Note that any $q\in U_{\Lambda,\kappa}$ is actually bounded 
in $BV$ since 
\[
\norme{q}_{BV(\Omega)}\leq 2\max(\Lambda,\kappa,|\alpha-1|).
\]
With this property at hand, we can get a similar result to Theorem \ref{thm:existence_min} without adding a penalization term in the cost function, hence $\beta=0$. 

\begin{theorem}\label{thm:existence_min_U_lambda_kappa}
Assume that the cost function $(q,\psi)\in U_{\Lambda,\kappa}\mapsto J(q,\psi)\in \R$ satisfies $(A2)-(A3)$ given in Theorem \ref{thm:existence_min} and that $\beta=0$. Then the optimization problem \eqref{eq:Optim_pbm} with $U=U_{\Lambda,\kappa}$ has at least one optimal solution.
\end{theorem}

\begin{proof}
We introduce the following non-empty set
\[
\mathcal{A}=\left\{(q,\psi)\in U_{\Lambda,\kappa}\times H^1(\Omega)\ |\ a(q;\psi,\phi)=b(q;\phi)\ \forall \phi\in H^1(\Omega) \right\}.
\]
From $(A2)$, $J(q,\psi)$ is bounded from below on $\mathcal{A}$. We thus have a minimizing sequence $(q_n,\psi_n)\in \mathcal{A}$ such that 
\[
\lim_{n\to +\infty} J(q_n,\psi_n)=\inf_{(q,\psi)\in \mathcal{A}} J(q,\psi).
\] 
Since $(q_n)_n\subset U_{\Lambda,\kappa}$, it satisfies $\norme{q_n}_{BV(\Omega)}\leq 2\max(\Lambda,\kappa,|\alpha-1|)$ and thus admits a convergent subsequence toward some $q\in U_{\Lambda,\kappa}$.
Theorem \ref{thm:CTS_mapping} then gives that $\psi(q_n)\to \psi(q)$ strongly in $H^1(\Omega)$ and the proof can be finished as the proof of Theorem \ref{thm:existence_min}.
\end{proof}

\section{Boundedness/Continuity of solution to Helmholtz problem}\label{sec:continuity}

In this section, we prove that even if the parameter $q$ is not smooth enough for the solution to \eqref{eq:Helmholtz} to be in $H^s(\Omega)$ for some $s>1$, we can still have continuous solution.
In order to prove such regularity for $\psi$, we are going to rely on the De Giorgi-Nash-Moser theory \cite[Chapter 8.5]{Trudinger}, \cite[Chapters 3.13, 7.2]{Ural_1968} and more precisely 
on \cite[Proposition 3.6]{Nittka_2011} which reads
\begin{theorem}\label{thm:Tool_for_boundedness}
Consider the elliptic problem associated with inhomogeneous Neu\-mann boundary condition given by
\begin{equation} \left\{
\begin{aligned}
\mathcal{L} v\vcentcolon=\div{A(x)\nabla v} 
	&= f_0-\sum_{j=1}^N \dpart{f_j}{x_j},\\ 
\dn{v} 
	&= h+\sum_{j=1}^N f_j n_j, 
\end{aligned}
\right.\label{eq:Elliptic_PDE}
\end{equation}
where $A\in L^\infty(\Omega,\R^{N\times N})$ satisfy the standard ellipticity condition $A(x)\xi\cdot\xi \geq \gamma|\xi|^2$ for a.e. $x\in\Omega$.
Let $p>N$ and assume that $f_0\in L^{p/2}(\Omega)$, $f_j\in L^p(\Omega)$ for all $j=1,\cdots,N$ and $h\in L^{p-1}(\partial\Omega)$. Then the weak solution $v$
to \eqref{eq:Elliptic_PDE} satisfies 
\[
\norme{v}_{C^0(\Omega)}\leq C(N,p,\Omega,\gamma)\left( \norme{v}_{L^2(\Omega)} +\norme{f_0}_{L^{p/2}(\Omega)} +\sum_{j=1}^N \norme{f_j}_{L^p(\Omega)}+\norme{h}_{L^{p-1}(\partial\Omega)}\right).
\]   
\end{theorem}

\subsection{$C^0$-bound for the general Helmholtz problem}
Using Theorem \ref{thm:Tool_for_boundedness}, we can prove some $L^\infty$ bound for the weak solution to Helmholtz equation with bounded coefficients.

\begin{theorem}\label{thm:Boundedness_Helmholtz}
Assume that $q\in L^\infty(\Omega)$ and satisfies \eqref{eq:hyp_q} and $g\in L^2(\partial\Omega)$. Then the solution to Problem \eqref{eq:FV_Helmholtz} satisfies 
\begin{equation}\label{eq:Final_L_infty_bound}
\norme{\psi}_{C^0(\Omega)}
	\leq \widetilde{C}(\Omega)\widetilde{C}_\mathrm{s}(k_0,\alpha)\left(\norme{q}_{L^\infty(\Omega)} \norme{\nabla \psi_0}_{L^\infty(\Omega)} +\norme{g}_{L^2(\partial\Omega)} \right),
\end{equation}  
where 
\[
\widetilde{C}_\mathrm{s}(k_0,\alpha) = 1 +\left((1+k_0^2)k_0^{-1} +\alpha^{-1/2}\right)\max\{k_0^{-1},\alpha^{-1/2}\}C_{\mathrm{s}}(k_0),
\]
and $\widetilde{C}(\Omega)>0$ does not depend on $k$ nor $q$.
\end{theorem}

\begin{proof}
We cannot readily apply Theorem \ref{thm:Tool_for_boundedness} to the weak solution of Problem \eqref{eq:Helmholtz} since it involves a complex valued operator. We therefore consider the Problem satisfied by $\nu=\Re{u}$ and $\zeta=\Im{u}$ which is given by 
\begin{equation}\left\{
\begin{aligned}
-\div{(1+q)\nabla \nu}-k_0^2\nu 
	&=\div{q\nabla \Re{\psi_0}}
	&&\textrm{ in } \Omega, \\
-\div{(1+q)\nabla \zeta}-k_0^2\zeta 
	&=\div{q\nabla \Im{\psi_0}}
	&&\textrm{ in } \Omega, \\
(1+q)\dn{\nu} 
	&=\Re{g}-k_0\zeta-q \dn{\Re{\psi_0}},		
	&&\textrm{ on } \partial\Omega, \\
(1+q)\dn{\zeta} 
	&=\Im{g}+k_0\nu-q\dn{\Im{\psi_0}}
	&&\textrm{ on }\partial\Omega.
\end{aligned}
\right.\label{eq:Helmholtz_real_imaginary}
\end{equation}
Since Problem \eqref{eq:Helmholtz_real_imaginary} is equivalent to Problem (\ref{eq:Helmholtz}), we get that the weak solution $(\nu,\zeta)\in H^1(\Omega)$ to (\ref{eq:Helmholtz_real_imaginary}) satisfies the inequality (\ref{eq:Uniform_H1_bound}). 
Assuming that $g\in L^2(\partial\Omega)$ and
using the continuous Sobolev embedding $H^1(\Omega)\subset L^6(\Omega)$, the (compact) embedding $H^{1/2}(\partial\Omega)\subset L^2(\partial\Omega)$, that $q\in L^\infty(\Omega)$ satisfies \eqref{eq:hyp_q} and the fact that $\psi_0$ is smooth we get the next regularities 
%
\begin{align*}
	f_{0,1}&=k_0^2\nu\in L^6(\Omega),\ 
	f_{j,1}=q\dpart{\Re{\psi_0}}{x_j}\in L^\infty(\Omega),\ 
	h_{1}=\Re{g}-k_0\zeta\in L^2(\partial\Omega), \\
	f_{0,2}&=k_0^2\zeta\in L^6(\Omega),\ 
	f_{j,2}=q\dpart{\Im{\psi_0}}{x_j}\in L^\infty(\Omega),\ 
	h_{2}=\Im{g}+k_0\nu\in L^2(\partial\Omega).
\end{align*}

Applying now Theorem \ref{thm:Tool_for_boundedness} to \eqref{eq:Helmholtz_real_imaginary} twice with $p=3$ and $N=2$, one 
gets $C^0$ bounds for $\nu$ and $\zeta$
\begin{align*} 
\norme{\nu}_{C^0(\Omega)}
	&\leq C(2,3,\Omega,\gamma)\left( \norme{\nu}_{L^2(\Omega)} +\norme{f_{0,1}}_{L^{3/2}(\Omega)} +\sum_{j=1}^2 \norme{f_{j,1}}_{L^3(\Omega)}+\norme{h_1}_{L^{2}(\partial\Omega)}\right),
\\
\norme{\zeta}_{C^0(\Omega)}
	&\leq C(2,3,\Omega,\gamma)\left( \norme{\zeta}_{L^2(\Omega)} +\norme{f_{0,2}}_{L^{3/2}(\Omega)} +\sum_{j=1}^2 \norme{f_{j,2}}_{L^3(\Omega)}+\norme{h_2}_{L^{2}(\partial\Omega)}\right).
\end{align*}

Some computations with the Holder and multiplicative trace inequalities then give
\begin{align*}
(\norme{\nu}_{L^2(\Omega)} +\norme{\zeta}_{L^2(\Omega)})
	&\leq 2 \norme{\psi}_{L^2(\Omega)},\\ 
\norme{f_{0,1}}_{L^{3/2}(\Omega)} +\norme{f_{0,2}}_{L^{3/2}(\Omega)}
	&\leq k_0^2 \norme{\psi}_{L^{3/2}(\Omega)} \leq |\Omega|^{1/6} k_0^2 \norme{\psi}_{L^{2}(\Omega)}, \\
\norme{f_{j,l}}_{L^3(\Omega)}
	&\leq \norme{q}_{L^\infty(\Omega)} \norme{\nabla \psi_0}_{L^\infty(\Omega)},\ j=1,2, \\
\norme{h_1}_{L^2(\partial\Omega)} +\norme{h_2}_{L^2(\partial\Omega)} 
	&\leq \norme{g}_{L^2(\partial\Omega)}+k_0\norme{\psi}_{L^2(\partial\Omega)} \\
	& \leq \norme{g}_{L^2(\partial\Omega)}+k_0C(\Omega)\sqrt{\norme{\psi}_{L^2(\Omega)}\norme{\psi}_{H^1(\Omega)}}.
\end{align*}
Using then Young's inequality yields
\begin{align*}
k_0\sqrt{\norme{\psi}_{L^2(\Omega)}\norme{\psi}_{H^1(\Omega)}} 
	&\leq C \left(\norme{\psi}_{H^1(\Omega)}+k_0^2\norme{\psi}_{L^2(\Omega)}\right)\\
	&\leq C \left(\norme{\nabla \psi}_{L^2(\Omega)}+k_0^2	\norme{\psi}_{L^2(\Omega)}\right) 
\end{align*}
where $C>0$ is a generic constant. We obtain the bound 
\begin{align*}
\norme{\psi}_{C^0(\Omega)}
	&= \norme{\nu}_{C^0(\Omega)}+\norme{\zeta}_{C^0(\Omega)} \\
	&\leq \widetilde{C}(\Omega)\left( \left(1+k_0^2 \right)\norme{\psi}_{L^2(\Omega)}+\norme{\nabla \psi}_{L^2(\Omega)} +\norme{q}_{L^\infty(\Omega)} \norme{\nabla \psi_0}_{L^\infty(\Omega)} +\norme{g}_{L^2(\partial\Omega)} \right). 
\end{align*}

Using the definition of $\norme{\psi}_{1,k_0}$ on the estimate above, we get
\begin{equation} \label{eq:L_infty_bound}
\begin{aligned} 
\norme{\psi}_{C^0(\Omega)}
	&\leq \widetilde{C}(\Omega)\Big( \left((1+k_0^2)k_0^{-1} +\alpha^{-1/2}\right)\norme{\psi}_{1,k_0} \\
	&\hspace{2cm} +\norme{q}_{L^\infty(\Omega)} \norme{\nabla \psi_0}_{L^\infty(\Omega)} +\norme{g}_{L^2(\partial\Omega)} \Big).
\end{aligned} 
\end{equation}  
To apply the a priori estimate \eqref{eq:Uniform_H1_bound}, we recall that the $H^{-1/2}$ norm can be replaced by a $L^2$ norm (since $g\in L^2(\partial\Omega)$) and then, 
\begin{align*}
\norme{\psi}_{1,k_0}
	&\leq C(\Omega)\max\{k_0^{-1},\alpha^{-1/2}\}C_{\mathrm{s}}(k_0)\left(\norme{q}_{L^\infty(\Omega)}\norme{\nabla \psi_0}_{L^2(\Omega)}+\norme{g}_{L^{2}(\partial\Omega)} \right) \\
	&\leq C(\Omega)\max\{k_0^{-1},\alpha^{-1/2}\}C_{\mathrm{s}}(k_0)\max\{1,\sqrt{\abs{\Omega}}\}\left(\norme{q}_{L^\infty(\Omega)}\norme{\nabla \psi_0}_{L^{\infty}(\Omega)}+\norme{g}_{L^{2}(\partial\Omega)} \right)	
\end{align*}

Finally, combining the \JS{latter} expression with Equation \eqref{eq:L_infty_bound}, we obtain that the weak solution to the Helmholtz equation satisfies
\begin{align*}
\norme{\psi}_{C^0(\Omega)}
	&\leq \widetilde{C}(\Omega)\left(1 +\left((1+k_0^2)k_0^{-1} +\alpha^{-1/2}\right)\max\{k_0^{-1},\alpha^{-1/2}\}C_{\mathrm{s}}(k_0)\right) \\
	& \hspace{2cm} \times \left(\norme{q}_{L^\infty(\Omega)} \norme{\nabla \psi_0}_{L^\infty(\Omega)} +\norme{g}_{L^2(\partial\Omega)} \right),
\end{align*}
where $\widetilde{C}(\Omega) > 0$.
\end{proof}

\begin{rem}
\begin{enumerate}
\item For the one-dimensional Helmholtz problem, the a priori estimate \eqref{eq:Uniform_H1_bound} and the continuous embedding $H^1(I)\subset C^0(I)$ directly gives 
the continuity of $u$ over \JS{a give interval} $I$ 
\[
\norme{\psi}_{C^0(I)}\leq C \norme{\psi}_{1,k_0}\leq C(k_0)\left(\norme{q}_{L^\infty(\Omega)}\norme{\nabla \psi_0}_{L^\infty(\Omega)}+\norme{g}_{H^{-1/2}(\partial\Omega)} \right).
\]
It is worth noting that we do not need to assume that $g\in L^2(\partial\Omega)$.
\item For the two-dimensional Helmholtz problem with $q=0$, we can get the above $\mathcal{C}^0$ estimate from the embedding 
$H^2(\Omega)\hookrightarrow \mathcal{C}^0(\overline{\Omega})$ since 
\[
\norme{\psi}_{C^0(\Omega)}\leq C \norme{\psi}_{H^2(\Omega)},
\]
for a generic constant $C$. We can then see that the estimate \eqref{eq:Final_L_infty_bound} \JS{has actually}  the same dependance with respect to $k_0$ as the $H^2$-estimate \JS{in}~\cite[p. 677, Proposition 3.6]{Hetmaniuk_2007}.  
\end{enumerate}
\end{rem}

\subsection{$C^0$-bounds for the total and scattered waves} 
Thanks to Remark \ref{rem:Link_general_Helmholtz_scatt} and following the proof of Theorem \ref{thm:Boundedness_Helmholtz}, these bounds can be roughly obtained by setting $g=\dn{\psi_0}-\ii k_0\psi_0$ and omitting the $L^{\infty}$-norms in \eqref{eq:L_infty_bound} for the total wave $\psi_{tot}$, and simply by setting $g=0$ in the case the scattered wave $\psi_{sc}$. Using after the $H^1$-bounds from Remark \ref{rem:H_1_bounds_total_sc_waves}, we actually get
\begin{align*}
\norme{\psi_{tot}}_{\mathcal{C}^0(\Omega)} 
	&\leq \widetilde{C}(\Omega)k_0 \left(\left((1+k_0^2)k_0^{-1} +\alpha^{-1/2}\right)\max\{k_0^{-1},\alpha^{-1/2}\}C_\mathrm{s}(k_0)+ 1\right)  \\
\norme{\psi_{sc}}_{\mathcal{C}^0(\Omega)} 
	&\leq \widetilde{C}(\Omega)k_0 \left(\left((1+k_0^2)k_0^{-1} +\alpha^{-1/2}\right)\alpha^{-1/2}C_\mathrm{s}(k_0) +1\right)\norme{q}_{L^\infty(\Omega)}. 
\end{align*}

We emphasize that the previous estimates show that the scattered wave $\psi_{sc}$ vanishes in $\Omega$ if $q\to 0$. This is expected since, if $q=0$, there is no obstacle to scatter the incident wave which \JS{amounts} to saying that $\psi_{tot}=\psi_0$.


\section{Discrete optimization problem and convergence results}\label{sec:disc_pb}

This section is devoted to the finite element discretization of the optimization problem \eqref{eq:Optim_pbm}. We consider a quasi-uniform family of triangulations (see \cite[p. 76, Definition 1.140]{Ern_book_FE}) $\left\{\mathcal{T}_h\right\}_{h>0}$ of $\Omega$ and the corresponding finite element spaces 
\begin{align*}
\mathcal{V}_h & =\left\{\phi_h\in \mathcal{C}(\overline{\Omega})\ |\ \phi_h|_{T}\in \mathbb{P}_1(T),\ \forall T\in\mathcal{T}_h\right\}. 
\end{align*}
Note that thanks to Theorem \ref{thm:Boundedness_Helmholtz}, the solution to the general Helmholtz equation \eqref{eq:Helmholtz} is continuous, which motivates to use continuous piecewise linear finite elements.
We are going to look for \JS{a} discrete optimal design that \JS{belongs} to some finite element spaces $\mathcal{K}_h$ and we thus introduce the 
following set of discrete admissible parameters
\[
U_{h}=U\cap \mathcal{K}_h.
\]
The full discretization of the optimization problem \eqref{eq:Optim_pbm} then reads
\begin{equation}\label{eq:Discrete_optim_pbm}
\textrm{Find } q^*_h\in U_{h}\ \textrm{such that } 
\widetilde{J}(q^*_h)\leq \widetilde{J}(q_h),\ \forall q_h\in U_{h},
\end{equation}
where $\widetilde{J}(q_h)=J(q_h,\psi_h(q_h))$ is the reduced cost-functional and $\psi_h\vcentcolon=\psi_h(q_h)\in \mathcal{V}_h$ satisfies the discrete Helmholtz problem
\begin{equation}\label{eq:Discrete_Helmholtz}
a(q_h;\psi_h,\phi_h)=b(q_h;\phi_h),\ \forall \phi_h\in  \mathcal{V}_h.
\end{equation}
The existence of solution to Problem \eqref{eq:Discrete_Helmholtz} \PH{is going to be discussed in the next subsection.}

Before giving the definition of $\mathcal{K}_h$, we would like to discuss briefly the strategy for proving that the discrete optimal solution 
converges toward the continuous ones. To achieve this, we need to pass to the limit in inequality \eqref{eq:Discrete_optim_pbm}.
Since $J$ is only lower-semi-continuous with respect to the weak$^*$ topology of $BV$, we can only pass to the limit on one side of the inequality and the
continuity of $J$ is then going to be needed to pass to the limit on the other side to keep this inequality valid as $h\to 0$. 
\newline
We discuss first the case $U=U_\Lambda$ for which Theorem \ref{thm:existence_min} gives the existence of optimal $q$ but only if $\beta>0$. Since we have to pass to the limit in 
\eqref{eq:Discrete_optim_pbm}, we need that $\Lim{h\to 0}|D q_h|(\Omega)=|D q|(\Omega)$. Since the total variation is only continuous with respect to the strong topology of $BV$, we have to approximate any $q\in U_\Lambda$ by some $q_h\in U_{h}$ such that 
\[
\lim_{h\to 0}\norme{q-q_h}_{BV(\Omega)}=0.
\]
However, from \cite[p. 8, Example 4.1]{Bartels_2012}
there \JS{exists} an example of a $BV$-function $v$ that cannot be approximated by piecewise constant function $v_h$ over a given mesh in such a way that 
$\Lim{h\to 0}|D v_h|(\Omega)=|D v|(\Omega)$.
Nevertheless, if one consider an adapted mesh that depends on a given function $v\in BV(\Omega)\cap L^\infty(\Omega)$, we get the 
existence of piecewise constant function on this specific mesh that strongly converges in $BV$ toward $v$ (see \cite[p. 11, Theorem 4.2]{Belik_2003}). As a result, 
when considering $U=U_\Lambda$, we use the following discrete set of admissible parameters 
\[
\mathcal{K}_{h,1}=\left\{q_h\in L^\infty(\Omega)\ |\ q_h|_{T}\in \mathbb{P}_1(T),\ \forall T\in\mathcal{T}_h\right\}.
\]
Note that, from Theorem \cite[p. 10, Theorem 4.1 and Remark 4.2]{Belik_2003}, the set $U_{h}=U_{\Lambda}\cap \mathcal{K}_{h,1}$ defined above has the required density property hence motivated its introduction as a discrete set of admissible parameter. 
\newline
In the case $U=U_{\Lambda,\kappa}$, we will not need the density of $U_h$ \JS{for} the strong topology of $BV$ but only for the weak$^*$ topology. The discrete set of admissible \JS{parameters} is then going to be 
$U_{h}=U_{\Lambda,\kappa}\cap \mathcal{K}_{h,0}$ with
\[
\mathcal{K}_{h,0}=\left\{q_h\in L^\infty(\Omega)\ |\ q_h|_{T}\in \mathbb{P}_0(T),\ \forall T\in\mathcal{T}_h\right\}.
\]

We show below the convergence of discrete optimal solution to the continuous one for both cases highlighted above. 

\subsection{Convergence of the Finite element approximation}

We prove here some useful approximations results for any $U_h$ defined above. We have the following convergence result whose proof can be found in \cite[p. 22, Lemma 4.1]{Esterhazy_2012} (see also \cite[p. 10, Theorem 4.1]{Graham_variable_2018}).

\begin{theorem}\label{thm:Cv_FE}
Let $q_h\in U_{h}$ and $\psi(q_h)\in H^1(\Omega)$ be the solution to the variational problem
\[
a(q_h;\psi(q_h),\phi)=b(q_h,\phi),\ \forall \phi\in H^1(\Omega).
\]
Let $S^*:(q_h,f)\in  U_{h}\times L^2(\Omega)\mapsto S^*(q_h,f)=\psi^*\in H^1(\Omega)$ be the solution operator associated to the following problem 
\[
\mathrm{Find\ }\psi^*\in H^1(\Omega)\ \mathrm{such\ that\ } a(q_h;\phi,\psi^*)=(\phi,\overline{f})_{L^2(\Omega)},\ \forall \phi\in H^1(\Omega).
\]
Denote by $C_a$ the continuity constant of the bilinear form $a(q_h;\cdot,\cdot)$, which does not depend on $h$ since $q_h\in  U_{h}$, and define the adjoint approximation property by
\[
\delta(\mathcal{V}_h)\vcentcolon=\sup_{f\in L^2(\Omega)} \inf_{\phi_h\in\mathcal{V}_h}\frac{\norme{S^*(q_h,f)-\phi_h}_{1,k_0}}{\norme{f}_{L^2(\Omega)}}.
\]
Assume that the spaces $\mathcal{V}_h$ satisfies 
\begin{equation}\label{hyp:maillage}
2C_a k_0\delta(\mathcal{V}_h)\leq 1,
\end{equation}
then the solution $\psi_h(q_h)$ to Problem (\ref{eq:Discrete_Helmholtz}) satisfies 
\[
\norme{\psi(q_h)-\psi_h(q_h)}_{1,k_0}\leq 2 C_a \inf_{\phi_h\in\mathcal{V}_h} \norme{\psi(q_h)-\phi_h}_{1,k_0}.
\]
\end{theorem}

\PH{We emphasize that the above error estimates in fact implies the existence and uniqueness of a solution to the discrete problem (\ref{eq:Discrete_Helmholtz}) (see \cite[Theorem 3.9]{lohndorf2011wavenumber}).}
In the case $q\in \mathcal{C}^{0,1}(\Omega)$ where $\Omega$ is a convex Lipschitz domain, \JS{Assumption}~(\ref{hyp:maillage}) has been discussed in \cite[p. 11, Theorem 4.3]{Graham_variable_2018} and roughly \JS{amounts} to say that (\ref{hyp:maillage}) holds if $k_0^2h$ is small enough.
Since the proof rely on $H^2$-regularity for a Poisson problem, we cannot readily extend the argument here since we can only expect to have $\psi\in H^1(\Omega)$ and that $S^*$ also depend on the meshsize. We can still show that \eqref{hyp:maillage} is satisfied for small enough $h$.

\begin{lemma}\label{lem:delta_V_h_0}
Assume that $q_h\in U_{h}$ weak$^*$ converges toward $q\in BV(\Omega)$. Then \eqref{hyp:maillage} is satisfied for small enough $h$.
\end{lemma} 

\begin{proof}
Note first that Theorem \ref{thm:CTS_mapping} also holds for the adjoint problem and thus 
\[
\lim_{h\to0}\norme{S^*(q_h,f)-S^*(q,f)}_{1,k_0}=0.
\]
Using the density of smooth \JS{functions} in $H^1$ and the properties of the piecewise linear interpolant \cite[p. 66, Corollary 1.122]{Ern_book_FE}, we have that 
\[
\lim_{h\to0}\left(\sup_{f\in L^2(\Omega)} \inf_{\phi_h\in\mathcal{V}_h}\frac{\norme{S^*(q,f)-\phi_h}_{1,k_0}}{\norme{f}_{L^2(\Omega)}}\right)=0,
\]
and thus a triangular inequality shows that \eqref{hyp:maillage} holds for small enough $h$.
\end{proof}

We can now prove a discrete counterpart to Theorem \ref{thm:CTS_mapping}.
\begin{theorem}\label{thm:CTS_discrete}
Let $(q_h)_h\subset U_{h}$ be a sequence satisfying $\norme{q_h}_{BV(\Omega)}\leq M$ and whose weak$^*$ limit in $BV(\Omega)$ is denoted by $q$. Let $(\psi_h(q_h))_h$ be the sequence of discrete solutions to Problem \eqref{eq:Discrete_Helmholtz}. Then $\psi(q_h)$ converges, as $h$ goes to $0$, strongly in $H^1(\Omega)$ towards $\psi(q)$ satisfying Problem \eqref{eq:FV_Helmholtz}.
\end{theorem}

\begin{proof}
For $h$ small enough, Lemma \ref{lem:delta_V_h_0} ensures that \eqref{hyp:maillage} holds and a triangular inequality then yields
\begin{align*}
\norme{\psi_h(q_h)-\psi(q)}_{1,k_0}
	&\leq \norme{\psi_h(q_h)-\psi(q_h)}_{1,k_0}+\norme{\psi(q_h)-\psi(q)}_{1,k_0}\\
	&\leq 2 C_a \inf_{\phi_h\in\mathcal{V}_h} \norme{\psi(q_h)-\phi_h}_{1,k_0}+\norme{\psi(q_h)-\psi(q)}_{1,k_0} \\
	&\leq (1+2C_a) \norme{\psi(q_h)-\psi(q)}_{1,k_0}+2 C_a \inf_{\phi_h\in\mathcal{V}_h} \norme{\psi(q)-\phi_h}_{1,k_0}.
\end{align*}
Theorem \ref{thm:CTS_mapping} gives that the first term above goes to zero as $h\to0$. For the second one, we can use the density of smooth function in $H^1$ to get that it goes to zero as well.
\end{proof}

\subsection{Convergence of the discrete optimal solution: Case $U_{h}=U_{\Lambda}\cap \mathcal{K}_{h,1}$}

We are now in \JS{a} position to prove the convergence of a discrete optimal design towards a continuous one in the case 
\[
U=U_\Lambda,\ U_h=U_{\Lambda}\cap \mathcal{K}_{h,1}.
\]
Hence the set of discrete control is composed of piecewise linear function on $\mathcal{T}_h$.


\begin{theorem}\label{thm:Cv_discrete_control}
Assume that $(A1)-(A2)-(A3)$ from Theorem \ref{thm:existence_min} hold and that the cost function $J_0:(q,\psi)\in U_\Lambda\times H^1(\Omega)\mapsto J_0	(q,\psi)\in \R$ is continuous with respect to the $(\mathrm{weak}^*,\mathrm{strong})$ topology of $BV(\Omega)\times H^1 (\Omega)$. Let $(q_h^*,\psi_h(q_h^*))\in U_{\Lambda,h}\times \mathcal{V}_h$ be an optimal pair of (\ref{eq:Discrete_optim_pbm}).
Then the sequence $(q^*_h)_h\subset U_\Lambda$ is bounded and there exists $q^*\in U_\Lambda$ such that $q^*_h\rightharpoonup q^*$ weakly$^*$ in $BV(\Omega)$, $\psi(q_h^*)\to \psi(q^*)$ strongly in $H^1(\Omega)$ and
\[
\widetilde{J}(q^*)\leq \widetilde{J}(q),\ \forall q\in U_{\Lambda}.
\]
Hence any accumulation point of $(q_h^*,\psi_h(q_h^*))$ is an optimal pair for Problem \eqref{eq:Optim_pbm}.
\end{theorem}

\begin{proof}
Let $q_{\Lambda}\in U_{\Lambda,h}$ be given as
\[
q_{\Lambda}(x)=\Lambda,\ \forall x\in\Omega. 
\]
Then $D q_{\Lambda}=0$. Since $\psi_h(q_{\Lambda})$ is well-defined and converges toward $\psi(q_{\Lambda})$ strongly in $H^1$ (see Theorem \ref{thm:Cv_discrete_control}), we have that 
\[
\widetilde{J}(q_{\Lambda})=J(q_{\Lambda},\psi_h(q_{\Lambda}))=J_0(q_\Lambda,\psi_h(q_{\Lambda})) \xrightarrow[h\to 0]{} J_0(q_\Lambda,\psi(q_{\Lambda})).
\]
As a result, using that $(q_h^*,\psi_h(q_h^*))$ is an optimal pair to Problem \eqref{eq:Discrete_Helmholtz}, we get that 
\[
\beta |D(q_h^*)|(\Omega) 
	\leq -J_0(q_h^*,\psi_h(q_h^*))+J(q_{\Lambda},\psi_h(q_{\Lambda}))		\leq -m +J_0(q_{\Lambda},\psi_h(q_{\Lambda})),
\]
and thus the sequence $(q_h^*)_h\subset U_{\Lambda,h}\subset U_{\Lambda}$ is bounded in $BV(\Omega)$ uniformly with respect to $h$. \JS{We can then assume that it converges and denote by} $q^*\in U_\Lambda$ its weak$^*$ limit and Theorem \ref{thm:CTS_discrete} then shows that $\psi_h(q_h^*)\to \psi(q^*)$ strongly in $H^1(\Omega)$. The lower semi-continuity of $J$ ensures that 
\[
J(q^*,\psi(q^*))
	=\widetilde{J}(q^*)\leq \liminf_{h\to 0}\widetilde{J}(q_h^*)
	=\liminf_{h\to 0}J(q_h^*,\psi_h(q_h^*)).
\]
Now, let $q\in U_\Lambda$, using the density of smooth \JS{functions} in $BV$, one gets that there exists a sequence $q_h\in U_{\Lambda,h}$ such that $\norme{q_h-q^*}_{BV(\Omega)}\to 0$ (see also \cite[p. 10, Remark 4.2]{Bartels_2012}). From Theorem \ref{thm:CTS_discrete}, one gets $\psi_h(q_h)\to \psi(q)$ strongly in $H^1(\Omega)$ and the continuity of $J$ ensure that 
$
\widetilde{J}(q_h)\rightarrow \widetilde{J}(q).
$
Since $\widetilde{J}(q_h^*)\leq \widetilde{J}(q_h)$ for all $q_h\in U_{\Lambda,h}$, one gets by passing to the inf-limit that
\[
\widetilde{J}(q^*)\leq \liminf_{h\to 0}\widetilde{J}(q_h^*)\leq \liminf_{h\to 0}\widetilde{J}(q_h) = \widetilde{J}(q),\ \forall q\in U_\Lambda,
\] 
and \JS{the proof is complete.}
\end{proof}

\subsection{Convergence of the discrete optimal solution: Case $U_{h}=U_{\Lambda,\kappa}\cap \mathcal{K}_{h,0}$}

We are now in \JS{a} position to prove the convergence of discrete optimal design toward continuous one in the case 
$$
U=U_{\Lambda,\kappa},\ U_h=U_{\Lambda,\kappa}\cap \mathcal{K}_{h,0}.
$$
Hence the set of discrete control is composed of piecewise constant \JS{functions} on $\mathcal{T}_h$ that satisfy 
$$
\forall q_h\in U_h,\ \norme{q_h}_{BV(\Omega)}\leq 2\max(\Lambda,\kappa,|\alpha-1|).
$$
We can compute explicitly the previous norm by integrating by parts the total variation (see e.g. \cite[p. 7, Lemma 4.1]{Bartels_2012}). 
This reads
$$
\forall q_h\in U_h,\ |Dq_h|(\Omega)=\sum_{F\in \mathcal{F}^i} |F|| [q_h]|_{F}|,
$$
where $\mathcal{F}^i$ is the set of interior faces and $| [q_h]|_{F}$ is the jump of $q_h$ on the interior face $F=\partial T_1\cap \partial T_2$ \JS{meaning that $| [q_h]|_{F}=|q_h|_{T_1}-|q_h|_{T_2}$, where $|\cdot|_{T_i}$ denotes the value of the a finite element function on the face $T_i$}.
Note then that any $q_h\in U_h$ can only have either a finite number of discontinuity or jumps that are not too large.

\begin{theorem}\label{thm:Cv_discrete_control_U_lambda_kappa}
Assume that $\beta=0$ and $(A2)-(A3)$ from Theorem \ref{thm:existence_min} hold and that the cost function $J:(q,\psi)\in U_\Lambda\times H^1(\Omega)\mapsto J(q,\psi)\in \R$ is continuous with respect to the $(\mathrm{weak}^*,\mathrm{strong})$ topology of $BV(\Omega)\times H^1 (\Omega)$. Let $(q_h^*,\psi_h(q_h^*))\in U_{h}\times \mathcal{V}_h$ be an optimal pair of \eqref{eq:Discrete_optim_pbm}.
Then the sequence $(q^*_h)_h\subset U_{\Lambda,\kappa}$ is bounded and there exists $q^*\in U_{\Lambda,\kappa}$ such that $q^*_h\rightharpoonup q^*$ weakly$^*$ in $BV(\Omega)$, $\psi(q_h^*)\to \psi(q^*)$ strongly in $H^1(\Omega)$ and
\[
\widetilde{J}(q^*)\leq \widetilde{J}(q),\ \forall q\in U_{\Lambda}.
\]
Hence any accumulation point of $(q_h^*,\psi_h(q_h^*))$ is an optimal pair for Problem \eqref{eq:Optim_pbm}.
\end{theorem}

\begin{proof}
Since $(q_h^*)_h$ belong to $U_h$, it satisfies $\norme{q_h}_{BV(\Omega)}\leq 2\max(\Lambda,\kappa,|\alpha-1|)$ and is thus bounded uniformly with respect to $h$. We denote by $q^*\in U_{\Lambda,\kappa}$
its weak$^*$ limit. Theorem \ref{thm:Cv_discrete_control} then shows that $\psi_h(q_h^*)$ converges strongly in $H^1(\Omega)$ toward $\psi(q^*)$. 
\newline
Now, let $q\in U_{\Lambda,\kappa}$, using the density of smooth function in $BV$, one gets that there exists a sequence $q_h\in U_{h}$ such that $q_h\rightharpoonup q$ weak$^*$ in $BV(\Omega)$ (see also \cite[Introduction]{Bartels_2012}). From Theorem \ref{thm:CTS_discrete}, one gets $\psi_h(q_h)\to \psi(q)$ strongly in $H^1(\Omega)$ and the continuity of $J$ ensure that 
$
\widetilde{J}(q_h)\rightarrow \widetilde{J}(q).
$
The proof can then be done as in Theorem \ref{thm:Cv_discrete_control}.
\end{proof}


\section{Numerical experiments}\label{sec:numerics}
In this section, we tackle numerically the optimization problem \eqref{eq:Optim_pbm}, when it is constrained to the total amplitude $\psi_{tot}$ described by \eqref{ampli-helmholtz}. We focus on two examples: a \textit{damping problem}, where the computed bathymetry optimally reduces the magnitude of the incoming waves; and an \textit{inverse problem}, in which we recover the bathymetry from the observed magnitude of the waves.

In what follows, we consider an incident plane wave $\psi_{0}(x)=\mathrm{e}^{\ii k_0x\cdot \vec{d}}$ propagating in the direction $\vec{d}=(0\;\; 1)^{\t}$, with 
\[ k_0 = \dfrac{\omega_0}{\sqrt{g z_0}},\, 
\omega_0	= \dfrac{2\pi}{T_0}, \, T_0=20, \, g=9.81, \, z_0=3.\]
For the space domain, we set $\Omega=[0,L]^2$, where 
$L = \frac{10\pi}{k_0}$.
We also impose a $L^\infty$-constraint on the variable $q$, namely that $q\geq -0.9$.

\subsection{Numerical methods}
We discretize the space domain by using a structured triangular mesh of 8192 elements, that is a space step of $\Delta x=\Delta y=8.476472$. 

For the discretization of $\psi_{sc}$, we use a $\mathbb{P}^1$-finite element method. The optimized parameter $q$ is discretized through a $\mathbb{P}^0$-finite element method. Hence, on each triangle, the approximation of $\psi_{sc}$ is determined by three nodal values, located at the edges of the triangle, and the approximation of $q$ is determined by one nodal value, placed at the center of gravity of the triangle.

On the other hand, we perform the optimization through a subspace trust-region method, based on the interior-reflective Newton method described in~\cite{algo1} and~\cite{algo2}. Each iteration involves the solving of a linear system using the method of preconditioned conjugate gradients, for which we supply the Hessian multiply function. The computations are achieved with \texttt{MATLAB} (version 9.4.0.813654 (R2018a)). 

{
\begin{rem}
We emphasize that the setting of our numerical experiments presented below does not meet all the assumptions of Theorems \ref{thm:Cv_discrete_control}
and \ref{thm:Cv_discrete_control_U_lambda_kappa} which state {the convergence of the
optimum of the discretized/discete problem toward the optimum of the
continuous one}. Indeed, regarding Theorem \ref{thm:Cv_discrete_control},
we do not consider discrete optimization parameters that are piecewise affine bounded functions and the cost functions considered does not have the regularization term $\beta|Dq|(\Omega)$ with $\beta>0$. Concerning Theorem \ref{thm:Cv_discrete_control_U_lambda_kappa} we look for $q_h$ that are bounded and piecewise constant but we did not demand that $|Dq_h|(\Omega)\leq \kappa$ for some $\kappa>0$.
Nevertheless, 
{we have observed in our numerical experiments that 
$|Dq^*_h|(\Omega)$ remains bounded when $h$ varies. }
 We can thus {conjecture} that Theorem \ref{thm:Cv_discrete_control_U_lambda_kappa} actually applies to the two test cases considered in this paper. 
\end{rem}
}

\subsection{Example 1: a wave damping problem}
We first consider the minimization of the cost functional
\[ J(q,\psi_{tot})=\dfrac{\omega_0^2}{2} \int_{\Omega_0}|\psi_{tot}(x,y)|^2dxdy,\]
where $\Omega_0=[\frac{L}{6},\frac{5L}{6}]^2$ is the domain where the waves are to be damped. The bathymetry is only optimized on a subset $\Omega_q=[\frac{L}{4},\frac{3L}{4}]^2\subset\Omega_0$.

The results are shown in Figure~\ref{topo_q} for the bathymetry and Figure~\ref{topo_u} for the wave. We observe that the optimal topography we obtain is highly oscillating.  In our experiments,  this oscillation remained at every level of space discretization we have tested. This could be related to the fact that in all our results, $q\in BV(\Omega)$. Note also that the damping is more efficient over $\Omega_q$. This fact is coherent with the results of the next experiment.

\begin{figure}[h!]
	\centering
	\begin{subfigure}{0.475\textwidth}
	\includegraphics[width=\linewidth]{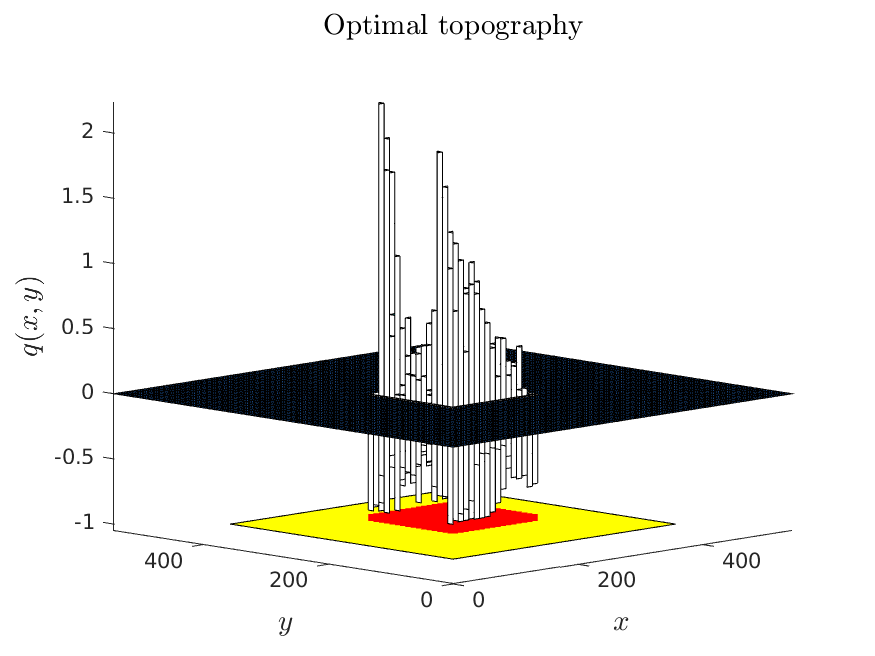}
	\caption{View from above.}
	\end{subfigure}
	\hfill
	\begin{subfigure}{0.475\textwidth}
\includegraphics[width=\linewidth]{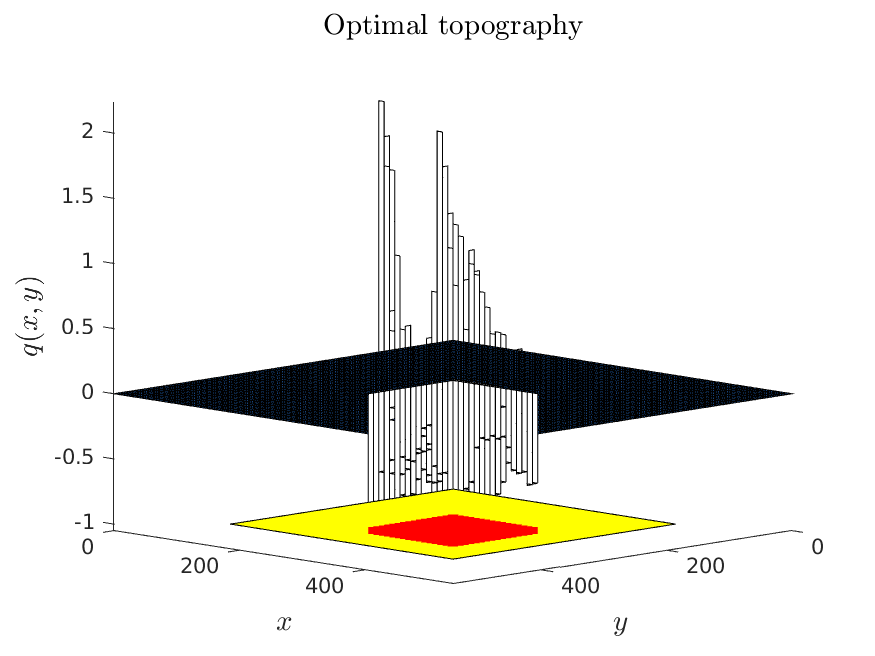}
  	\caption{View from below.}
	\end{subfigure}
	\caption[Optimal bathymetry for a wave damping problem]{Optimal topography for a wave damping problem. The yellow part represents $\Omega_0$ and the red part corresponds to the nodal points associated with $q$.
}\label{topo_q}
\end{figure}

\begin{figure}[h!]
	\centering
	\begin{subfigure}{0.75\textwidth}
	\includegraphics[width=\linewidth]{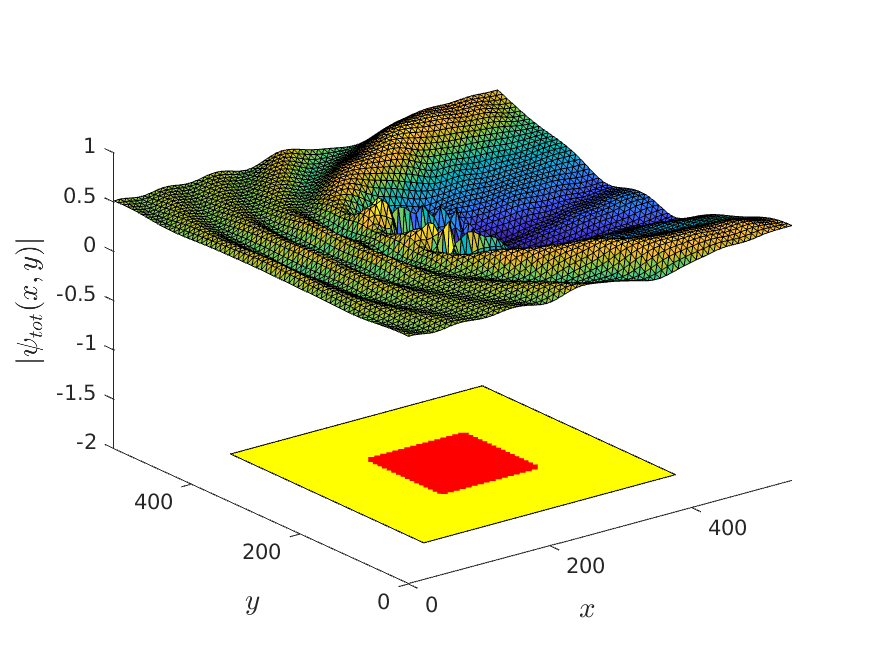}
  	\caption{Norm of the numerical solution.}
	\end{subfigure}	
	\vskip\baselineskip
	\begin{subfigure}{0.475\textwidth}
	\includegraphics[width=\linewidth]{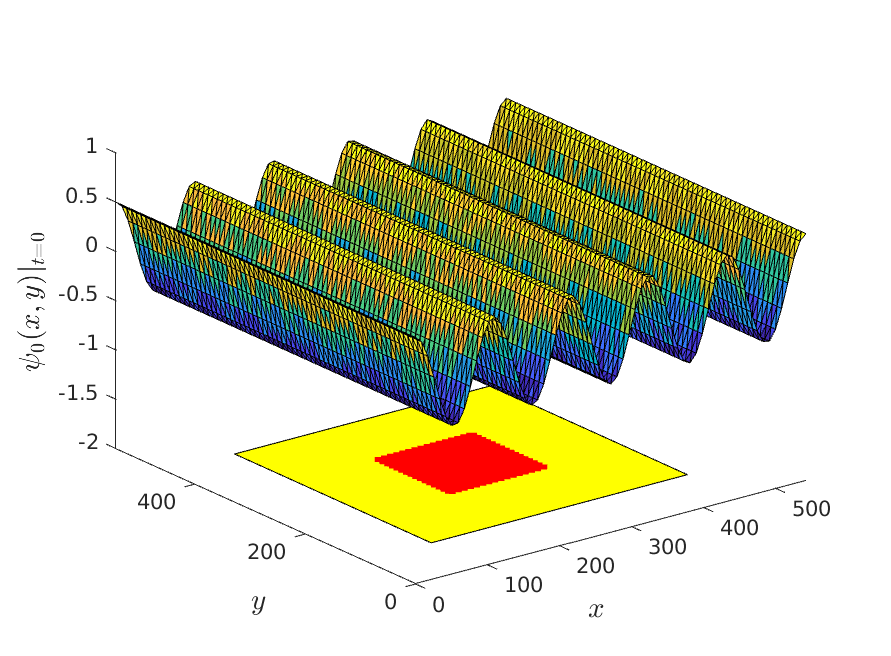}
	\caption{Real part of the incident wave.}
	\end{subfigure}
	\hfill
	\begin{subfigure}{0.475\textwidth}
	\includegraphics[width=\linewidth]{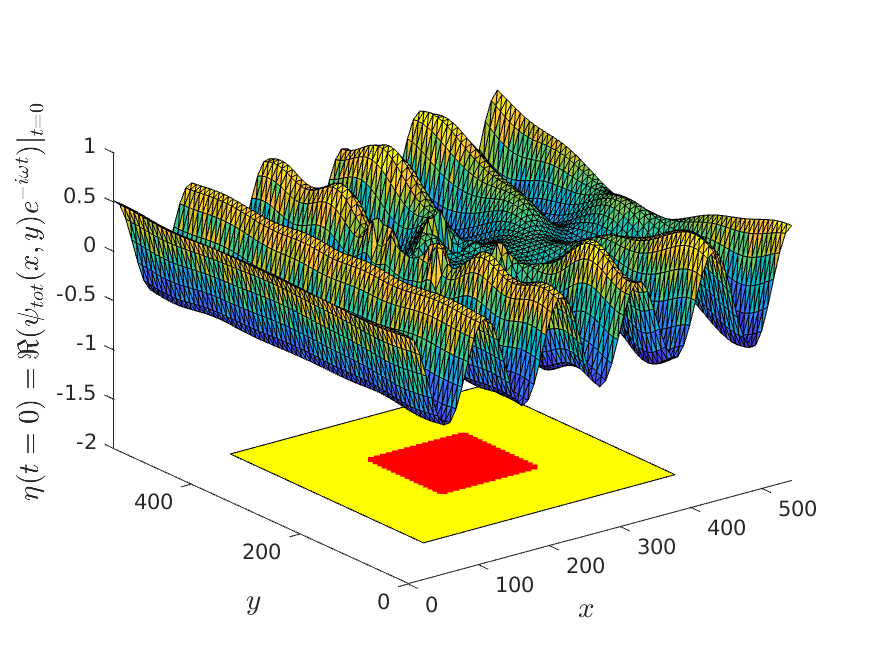}
  	\caption{Real part of the numerical solution.}
	\end{subfigure}
\caption[Numerical solution of a wave damping problem]{Numerical solution of a wave damping problem. The yellow part represents $\Omega_0$ and the red part corresponds to the nodal points associated with $q$.}\label{topo_u}	
\end{figure}

\subsection{Example 2: an inverse problem}
Many inverse problems associated to Helmholtz equation have been studied in the literature. We refer for example to~\cite{Colton,Dorn,Thompson} and the references therein. Note that in most of these papers the inverse problem rather consists in determining the location of a scatterer or its shape, often meaning that $q(x,y)$ is assumed to be constant inside and outside it. On the contrary, the inverse problem we consider in this section consists in determining a full real valued function.

Given the bathymetry 
\[ 
q_{ref}(x,y)	\vcentcolon= \e{-\tau\left(((x-\tfrac{L}{4})^2+(y-\tfrac{L}{4})^2\right)}+\e{-\tau\left((x-\tfrac{3L}{4})^2+(y-\tfrac{3L}{4})^2\right)},
\]
where $\tau= 10^{-3}$, we try to reconstruct it on the domain $\Omega_q = [\frac{L}{8},\frac{3L}{8}]^2\cup[\frac{5L}{8},\frac{7L}{8}]^2$, by minimizing the cost functional 
\[ 
J(q,\psi_{tot})=\dfrac{\omega_0^2}{2}\int_{\Omega_0}|\psi_{tot}(x,y)-\psi_{ref}(x,y)|^2dxdy,
\]
where $\psi_{ref}$ is the amplitude associated with $q_{ref}$ and $\Omega_0 = [\frac{3L}{4}-\delta,\frac{3L}{4}+\delta]^2$, $\delta=\frac{L}{6}$. Note that in this case, $\Omega_q$ is not contained in $\Omega_0$.

In Figure~\ref{inv_prob}, we observe that the part of the bathymetry that does not belong to the observed domain $\Omega_0$ is not recovered by the procedure. On the contrary, the bathymetry is well reconstructed in the part of the domain corresponding to $\Omega_0$. 

\begin{center}
\begin{figure}[h!]
	\centering
	\begin{subfigure}{0.75\textwidth}
	\includegraphics[width=\linewidth]{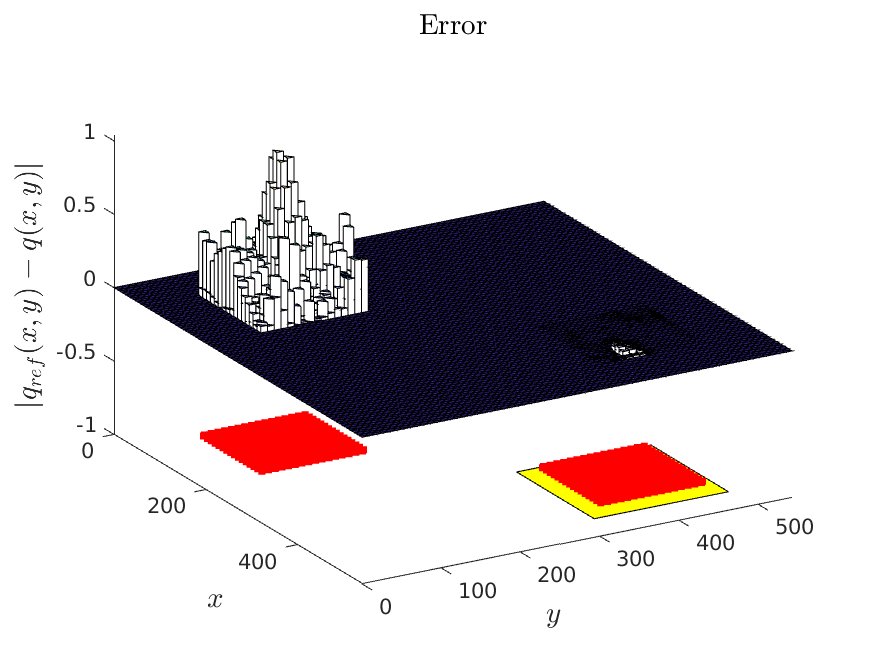}
  	\caption{Reconstruction error.}
	\end{subfigure}	
	\\
	\vskip\baselineskip
	%
%
	\begin{subfigure}{0.475\textwidth}
	\includegraphics[width=\linewidth]{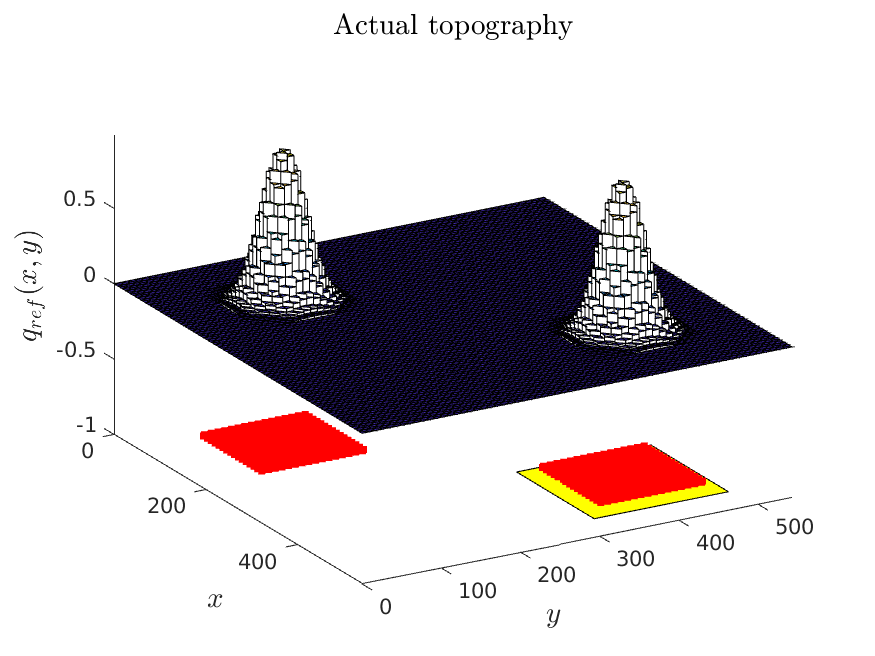}
	\caption{Actual bathymetry.}
	\end{subfigure}
	\hfill
	\begin{subfigure}{0.475\textwidth}
	\includegraphics[width=\linewidth]{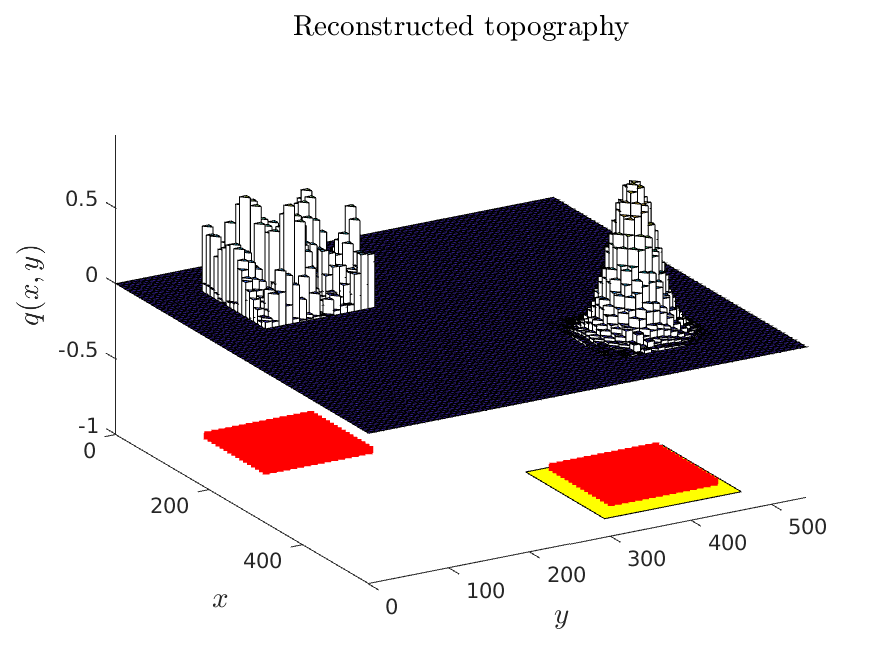}
  	\caption{Reconstructed bathymetry.}
	\end{subfigure}
\caption[Detection of a bathymetry from a wavefield]{Detection of a bathymetry from a wavefield. The yellow part represents $\Omega_0$ and the red part corresponds to the nodal points associated with $q$.}\label{inv_prob}
\end{figure}
\end{center}


%

\section*{Acknowledgments}
The authors acknowledge support from ANR Cin\'e-Para (ANR-15-CE23-0019) and ANR Allowap.

\section*{Appendix: derivation of Saint-Venant system}
For the sake of completeness and following the standard procedure described in~\cite{Perthame} (see also \cite{BSM,Sainte-Marie}), we derive the Saint-Venant equations from the Navier-Stokes system. 
For simplicity of presentation, system \eqref{Euler} is restricted to two dimensions, but a more detailed derivation of the three-dimensional case can be found in \cite{Decoene}. Since our analysis focuses on the shallow water regime, we introduce the parameter $\varepsilon := \dfrac{H}{L}$,  
where $H$ denotes the relative depth and $L$ is the characteristic dimension along the horizontal axis. The importance of the nonlinear terms is represented by the ratio $\delta := \dfrac{A}{H}$,
with $A$ the maximum vertical amplitude. We \JS{then} use the change of variables
\[
x' := \dfrac{x}{L},\, z' := \dfrac{z}{H},\, t' := \dfrac{C_0}{L}t,\]
and 
\[
u' := \dfrac{u}{\delta C_0},\, w' := \dfrac{w}{\delta\varepsilon C_0},\, 
\eta' := \dfrac{\eta}{A},\, z_b' := \dfrac{z_b}{H},\, 
p' := \dfrac{p}{gH}.
\]
where $C_0 = \sqrt{gH}$ is the characteristic dimension for the horizontal velocity. Assuming the viscosity and atmospheric pressure \JS{to be} constants, we define their respective dimensionless versions by 
\[
\mu' := \dfrac{\mu}{C_0 L },\, p_a' := \dfrac{p_a}{gH}.
\]
Dropping primes after rescaling, the dimensionless system \eqref{Euler} reads
\begin{align}
\delta\dpart{u}{t} +\delta^2\left(u\dpart{u}{x} +w\dpart{u}{z}\right)
	&= -\dpart{p}{x} +2\delta\dpp{\mu\dpart{u}{x}}{x}, \label{momentum} \\
	&\qquad +\delta\dpp{\mu\Big(\dfrac{1}{\varepsilon^2}\dpart{u}{z}+\dpart{w}{x}\Big)}{z} \nonumber
	\\
\varepsilon^2\delta\left(\dpart{w}{t} +\delta\Big(u\dpart{w}{x} +w\dpart{w}{z}\Big)\right)
	&= -\dpart{p}{z} -1  \label{vert-velocity}\\
	&\qquad +\delta\dpp{\mu\Big(\dpart{u}{z} +\varepsilon^2\dpart{w}{x}\Big)}{x} +2\delta\dpp{\mu\dpart{w}{z}}{z}, \nonumber \\
\dpart{u}{x} +\dpart{w}{z} &= 0. \label{mass}
\end{align}
The boundary conditions in \eqref{flux} remains similar and reads
\begin{equation}\left\{
\begin{aligned}
-\delta u\dpart{\eta}{x} +w 
	&= \dpart{\eta}{t}\sqrt{1+(\varepsilon\delta)^2 \abs{\dpart{\eta}{x}}^2}
	&&\textrm{on } (x,\delta\eta(x,t),t),\\
u\dpart{z_b}{x} +w 
	&= 0 
	&&\textrm{on } (x,-z_b(x),t).
\end{aligned} \right.\label{flux-2}
\end{equation}
However, the rescaled boundary conditions in \eqref{stress} are now given by  
\begin{align}
\left(p -2\delta\mu\dpart{u}{x}\right)\dpart{\eta}{x} +\mu \left(\dfrac{1}{\varepsilon^2}\dpart{u}{z} +\dpart{w}{x}\right)	
	&= p_a\dpart{\eta}{x} 
	&&\textrm{on } (x,\delta\eta(x,t),t), \label{atm-pressure-1} \\
\delta^2\mu\left(\dpart{u}{z} +\varepsilon^2\dpart{w}{x}\right)\dpart{\eta}{x}	
+\left(p -2\delta\mu\dpart{w}{z}\right) 
	&= p_a 
	&&\textrm{on } (x,\delta\eta(x,t),t),\label{atm-pressure-2}
\end{align}
and at the bottom $(x,-z_b(x),t)$:
\begin{equation}
\begin{aligned}
\varepsilon\left(p-2\delta\mu\dpart{u}{x}\right)\dpart{z_b}{x} 
+\delta\mu\left(\dfrac{1}{\varepsilon}\dpart{u}{z} +\varepsilon\dpart{w}{x}\right)\hspace{2cm}
	&\\
-\delta\mu\Big(\dpart{u}{z} +\varepsilon^2\dpart{w}{x}\Big)\left(\dpart{z_b}{x}\right)^2  
+\varepsilon\left(2\delta\mu\dpart{w}{z}-p\right)\dpart{z_b}{x}
	&= 0.
\end{aligned} \label{bottom-pressure} 
\end{equation}
\smallskip

To derive the Saint-Venant equations, we use an asymptotic analysis in $\varepsilon$. In addition, we assume a small viscosity coefficient
\[
\mu = \varepsilon\mu_0.
\]  
A first simplification of the system consists in deriving an explicit expression for $p$, known as the \textit{hydrostatic pressure}. Indeed, after rearranging the terms of order $\varepsilon^2$ in \eqref{vert-velocity} and integrating in the vertical direction, we get
\begin{align}
p(x,z,t)
	&= \bigO(\varepsilon^2\delta) +(\delta\eta -z) +\varepsilon\delta\mu_0\left(\dpart{u}{x} +2\dpart{w}{z}-\dpart{u}{x}(x,\eta,t)\right)\nonumber \\
	&\qquad +p(x,\delta\eta,t)-2\varepsilon\delta\mu_0\dpart{w}{z}(x,\eta,t). \label{pressure} 
\end{align}
To compute explicitly the last term, we combine \eqref{atm-pressure-1} with \eqref{atm-pressure-2} to obtain
\begin{align*}
p(x,\delta\eta,t) -2\varepsilon\delta\mu_0\dpart{w}{z}(x,\delta\eta,t)
	&= p_a\left(1 -(\varepsilon\delta)^2\Big(\dpart{\eta}{x}\Big)^2\right) \\
	&\qquad +(\varepsilon\delta)^2\left(p-2\varepsilon\mu_0\dpart{u}{x}(x,\eta,t)\right)\left(\dpart{\eta}{x}\right)^2,
\end{align*}
that can be combined with~\eqref{pressure} to obtain
\begin{equation} 
p(x,z,t) = (\delta\eta -z) +p_a +\bigO(\varepsilon\delta). \label{hydro-pressure}
\end{equation}
As a second approximation, we integrate vertically equations \eqref{mass} and \eqref{momentum}. We introduce $h_{\delta} = \delta\eta +z_b$. Due to the Leibnitz integral rule and the boundary conditions in \eqref{flux-2}, integrating the mass equation \eqref{mass} gives
\begin{align*}
\int_{-z_b}^{\delta\eta}\left(\dpart{u}{x} +\dpart{w}{z}\right)dz &= 0\\
\dpp{\int_{-z_b}^{\delta\eta}u dz}{x} -\delta u(x,\delta\eta,t)\dpart{\eta}{x} -u(x,-z_b,t)\dpart{z_b}{x} +w(x,\delta\eta,t) -w(x,-z_b,t) &= 0 \\
\dpart{\eta}{t}\sqrt{1+(\varepsilon\delta)^2 \abs{\dpart{\eta}{x}}^2} +\dpart{(h_{\delta}\overline{u})}{x} &= 0.
\end{align*}
To treat the momentum equation \eqref{momentum}, we notice that Equation \eqref{mass} allows us to rewrite the convective acceleration terms as
\[
u\dpart{u}{x} +w\dpart{u}{z} = \dpart{u^2}{x} +\dpart{uw}{z}.
\]
Its integration, combined with the boundary conditions in \eqref{flux-2}, leads to
\begin{align*}
\int_{-z_b}^{\delta\eta}\left(u\dpart{u}{x} +w\dpart{u}{z}\right)dz
	&= \dpp{\int_{-z_b}^{\delta\eta}u^2 dz}{x} -\delta u^2(x,\delta\eta,t)\dpart{\eta}{x} -u^2(x,-z_b,t)\dpart{z_b}{x} \\
	&\qquad +u(x,\delta\eta,t)\cdot w(x,\delta\eta,t) -u(x,-z_b,t)\cdot w(x,-z_b,t)\\
	&= \dpart{(h_{\delta}\overline{u^2})}{x} +u(x,\delta\eta,t)\dpart{\eta}{t}\sqrt{1+(\varepsilon\delta)^2\abs{\dpart{\eta}{x}}^2},
\end{align*}
where we have introduced the depth-averaged velocity 
\[
\overline{u}(x,t) := \dfrac{1}{h_{\delta}(x,t)}\int_{-z_b}^{\delta\eta}u(x,z,t)dz.
\]
The vertical integration of the left-hand side of \eqref{momentum} then brings
\begin{align*}
\int_{-z_b}^{\delta\eta}\left[\delta\dpart{u}{t} +\delta^2\left(u\dpart{u}{x} +w\dpart{u}{z}\right)\right]dz
	&= \delta\dpart{(h_{\delta}\overline{u})}{t} +\delta^2\dpart{(h_{\delta}\overline{u^2})}{x}\\ 
	&\qquad +\delta^2 u(x,\delta\eta,t)\dpart{\eta}{t}\bigg(\sqrt{1+(\varepsilon\delta)^2\abs{\dpart{\eta}{x}}^2} -1\bigg).
\end{align*}
To deal with the term $h_{\delta}\overline{u^2}$, we start from \eqref{hydro-pressure} which shows that $\dfrac{\partial p}{\partial x} = \bigO(\delta)$. Plugging this expression into \eqref{momentum} yields
\[
\dfrac{\partial^2 u}{\partial z^2} = \bigO(\varepsilon).
\]
From boundary conditions \eqref{atm-pressure-1} and \eqref{bottom-pressure}, we obtain
\[
\dpart{u}{z}(x,\delta\eta,t) = \bigO(\varepsilon^2),\; 
\dpart{u}{z}(x,z_b,t) = \bigO(\varepsilon).
\] 
Consequently, $u(x,z,t) = u(x,0,t) +\bigO(\varepsilon)$ and then $u(x,z,t)- \overline{u}(x,t) = \bigO(\varepsilon)$. 
Hence, we have the approximation
\[
h_\delta\overline{u^2}
= h_\delta\overline{u}^2 +\int_{-z_b}^{\delta\eta}(\overline{u}-u)^2 dz
= h_\delta\overline{u}^2 +\bigO(\varepsilon^2)
\]
and finally 
\begin{align}
\int_{-z_b}^{\delta\eta}\left[\delta\dpart{u}{t} +\delta^2\left(u\dpart{u}{x} +w\dpart{u}{z}\right)\right]dz
	&= \delta\dpart{(h_{\delta}\overline{u})}{t} +\delta^2\dpart{(h_{\delta}\overline{u}^2)}{x} +\bigO(\varepsilon^2\delta^2)\nonumber \\
	&\qquad +\delta^2 u(x,\delta\eta,t)\dpart{\eta}{t}\bigg(\sqrt{1+(\varepsilon\delta)^2\abs{\dpart{\eta}{x}}^2}-1\bigg).  \label{momentum-av-LHS}
\end{align}

We then integrate the right-hand side of Equation \eqref{momentum}
\begin{align*}
\int_{-z_b}^{\delta\eta}\bigg[-\dpart{p}{x} 
	& +\delta\dfrac{\mu_0}{\varepsilon}\dpp{\dpart{u}{z}}{z}
	+\varepsilon\delta\mu_0\left(2\dpp{\dpart{u}{x}}{x} 
	+\dpp{\dpart{w}{x}}{z}\right)\bigg]dz \\
	&= -\delta h_{\delta}\dpart{\eta}{x} +\bigO(\varepsilon\delta) 
	+\delta\left[\dfrac{\mu_0}{\varepsilon}\dpart{u}{z}(x,\delta\eta,t) -\dfrac{\mu_0}{\varepsilon}\dpart{u}{z}(x,-z_b,t)\right].
\end{align*}
Combining this expression with \eqref{momentum-av-LHS}, we get the vertical integration of the momentum equation:
\begin{align}
\dpart{\eta}{t}\sqrt{1+(\varepsilon\delta)^2 \abs{\dpart{\eta}{x}}^2} +\dpart{(h_{\delta}\overline{u})}{x} 
	&= 0 \label{mass-av} \\
\dpart{(h_{\delta}\overline{u})}{t} +\delta\dpart{(h_{\delta}\overline{u}^2)}{x}
	&= -h_{\delta}\dpart{\eta}{x}  
+\left[\dfrac{\mu_0}{\varepsilon}\dpart{u}{z}(x,\delta\eta,t) -\dfrac{\mu_0}{\varepsilon}\dpart{u}{z}(x,-z_b,t)\right] \nonumber \\
	&\qquad +\delta u(x,\delta\eta,t)\dpart{\eta}{t}\bigg(\sqrt{1+(\varepsilon\delta)^2\abs{\dpart{\eta}{x}}^2}-1\bigg)+\bigO(\varepsilon),\label{momentum-av} 
\end{align}
The convergence of \eqref{momentum-av} is guaranteed by the boundary equations \eqref{atm-pressure-1} and \eqref{bottom-pressure}, from which we get
\[
\dfrac{\mu_0}{\varepsilon}\dpart{u}{z}(x,\delta\eta,t) = \bigO(\varepsilon\delta),\; 
\dfrac{\mu_0}{\varepsilon}\dpart{u}{z}(x,-z_b,t) = \bigO(\varepsilon). \]
Hence the system~(\ref{mass-av2}--\ref{momentum-av2}).

\bibliographystyle{abbrv}			
\bibliography{bathymetry_bib}

\end{document}